\documentclass[final,11pt,3p]{elsarticle}

\usepackage{graphicx}
\usepackage{amsmath}
\usepackage{paralist}
\usepackage{graphics}
\usepackage{graphicx}
\usepackage{epstopdf}
\usepackage{amssymb}
\usepackage{color}
\usepackage{multirow}
\usepackage{booktabs}
\usepackage{bm}
\usepackage{mathrsfs}

\usepackage{amsmath}  
\usepackage{amssymb}   
\usepackage{amsthm}

\newcommand{\mc}[1]{{\mathcal #1}}
\newcommand{\mb}[1]{{\mathbf #1}}

\newtheorem{theorem}{Theorem}

\newtheorem{proposition}{Proposition}
\newtheorem{remark}{Remark}

\begin{document}
\begin{frontmatter}

\title{Mean Curvature Flow of Closed Curves Evolving in Two Dimensional Manifolds}

\author[CVUT]{Miroslav Kol\'a\v{r}}
\ead{kolarmir@fjfi.cvut.cz}

\author[UK]{Daniel \v{S}ev\v{c}ovi\v{c}\corref{Corr}}
\cortext[Corr]{Corresponding author}
\ead{sevcovic@fmph.uniba.sk}

\address[CVUT]{Department of Mathematics, Faculty of Nuclear Sciences and Physical Engineering Czech Technical University in Prague, Trojanova 13, Prague, 12000, Czech Republic}

\address[UK]{Comenius University Bratislava,
Department of Applied Mathematics and Statistics, Faculty of Mathematics, Physics and Informatics, Comenius University in Bratislava, Mlynska dolina, 842 48 Bratislava, Slovakia}

\begin{abstract}

We investigate the motion of a family of closed curves evolving according to the geometric evolution law on a given two dimensional manifold which is embedded or immersed in the three-dimensional Euclidean space. We derive a system of nonlinear parabolic equations describing the motion of curves belonging to a given two-dimensional manifold. Using the abstract theory of analytic semiflows, we prove the local existence, uniqueness of H\"older smooth solutions to the governing system of nonlinear parabolic equations for the position vector parametrization of evolving curves. We apply the method of flowing finite volumes in combination with the methods of lines for numerical approximation of the governing equations. Qualitative analytical results are illustrated by various numerical experiments.

\end{abstract}


\end{frontmatter}

\pagestyle{myheadings}
\thispagestyle{plain}
\markboth{M. KOL{\'A}\v{R}, D. \v{S}EV\v{C}OVI\v{C}}{Mean Curvature Flow of Closed Curves}

\section{Introduction}

In this article, we investigate the motion of a family $\{\Gamma_t, t\ge 0\}$ of closed curves evolving in the three-dimensional Euclidean space  $\mathbb{R}^3$ according to 
a subclass of the following geometric evolution law: 
 \begin{equation}
\partial_t\mb{X} = v_N \mb{N} + v_B \mb{B}  +  v_T \mb{T},
\label{eq:general}
\end{equation}
where the unit tangent $\mb{T}$, normal $\mb{N}$ and binormal $\mb{B}$ vectors form moving Frenet frame. In this paper, we restrict our interest to the investigation of the dynamics of three-dimensional closed curves on embedded and immersed manifolds in $\mathbb{R}^3$. 
 
The three-dimensional motion of closed curves is motivated by various physical applications arising in materials sciences, fluid dynamics, or molecular biology. In fluid dynamics, the motion of a curve in space is often applied to the analysis of vortex structures first studied by Helmholtz \cite{Helmholtz1858}. For a detailed summary of recent advances in the analysis of three-dimensional motion of curves,  we refer the reader to Kol{\'a}{\v{r}}, Bene\v{s} and \v{S}ev\v{c}ovi\v{c} \cite{kolar2022} and references therein. 

Among many important physical and engineering applications of flows of space curves evolving on prescribed surfaces, we mention the dislocation dynamics. A geodesic description of space curves is a convenient mathematical framework to simulate a dislocation cross-slip phenomenon in the crystalline structure of solids (cf. Kol{\' a}{\v r} \emph{et al.} \cite{PBKK:21}). The dislocation curves evolve on a prescribed two-dimensional manifold in $\mathbb{R}^3$.  Moreover, the dynamics of dislocation climbs \cite{niu2019dislocation,niu2017dislocation} investigated by Niu \emph{et al.} motivates us to study the problem of diffusion and transport along a moving space curve (cf.  Bene{\v s} \emph{et al.} \cite{benes2024diffusion}).
In nanomaterials manufacturing, a procedure called electrospinning is frequently used, i.e., by jetting polymer solutions in high electric fields into ultrafine nanofibers. The extruded fluid forms a curve evolving on the so-called Taylor conic surface. The jet oriface (outlet) is the vertex of the Taylor cone (cf. Reneker \cite{Reneker}, Yarin \emph{et al.} \cite{Yarin}, He \emph{et al.} \cite{He} and particular references therein). 
Structures arising from the electrospinning procedure move in space according to (\ref{eq:general}) and under the effect of electric forces \cite{Xu} to form nanofibers.
Recently, Kol{\' a}{\v r} and {\v S}ev{\v c}ovi{\v c} \cite{Kolar_algoritmy} analyzed, from the perspective of local existence and uniqueness, a motion of systems of space curves in the normal and binormal directions, mutually coupled by a Biot–Savart-type interaction. In \cite{kolar2017area}, an area-preserving flow was subsequently investigated using a geodesic formulation of the curves. The structure of the Helfrich flow of curves with non-local constraints was recently investigated by Kenmochi, Miyatake, and Sakakibara in \cite{kemmochi2024structure}. 

In \cite{Deckelnick2025}, Deckelnick and Nürnberg proposed a novel framework for the evolution of parametric curves driven by the anisotropic curve shortening flow in $\mathbb{R}^3$. Furthermore, in \cite{Binz}, Binz investigates optimal error estimates for semidiscrete and fully discrete schemes that approximate isotropic curve shortening flow in three dimensions. Their formulation relies on a careful choice of the tangential component of the velocity in parameterization, which transforms the evolution problem into a strictly parabolic differential equation. This equation is written in divergence form, enabling the construction of a natural variational discretization. For a fully discrete finite element scheme based on piecewise linear elements, optimal error bounds are established. Numerical experiments support the theoretical findings and illustrate the effectiveness of the method.

The paper is organized as follows. In the second section, a parametric description of evolving curves is introduced. In the third section, we restrict our interest to the motion on a closed surface without boundary. We derive a force term that attaches the curve to the surface and formulate the system of governing equations for such a restricted 3D motion. In the fourth section, we discuss the conditions for the existence and uniqueness of classical H\"older solutions. We also recall the role of the tangential velocity when using the parametric description. In the fifth section, we propose the numerical approximation scheme based on the flowing finite-volume approach. Finally, several computational examples are presented in the sixth section.

\section{Lagrangian description of evolving curves}

Our methodology for solving (\ref{eq:general}) is based on the so-called direct Lagrangian approach investigated by Dziuk \cite{dziuk1994convergence}, Deckelnick \cite{deckelnick1997weak}, Gage and Hamilton \cite{gage1986heat}, Mikula and \v{S}ev\v{c}ovi\v{c} \cite{sevcovic2001evolution, MS2004, MMAS2004, MS3}, and references therein. We explore the direct Lagrangian approach for an analytical and numerical solution of the geometric motion law (\ref{eq:general}). The evolving family of curves $\Gamma_t$ is parametrized by smooth mapping $\mb{X}:I\times[0,\infty) \to \mathbb{R}^3$, such that $\Gamma_t = \{ \mb{X}(u,t), u\in I\}, t\ge 0$. In what follows, we denote by $I=\mathbb{R}/\mathbb{Z}\simeq S^1$ the periodic interval $I=[0,1]$ isomorphic to the unit circle $S^1$. We assume that the scalar velocities $v_N, v_T, v_B$ are smooth functions of the position vector $\mb{X}\in\mathbb{R}^3$, the curvature $\kappa$, the torsion $\tau$,  and possibly depend nonlocally on the quantities associated with the curve $\Gamma_t$ itself, for example, its length. 
That is,
\[
v_K= v_K(\mb{X}, \kappa, \tau, \mb{T}, \mb{N}, \mb{B}, \Gamma_t), \quad K\in\{T, N, B\}.
\]
The unit tangent vector $\mb{T}$ to $\Gamma_t$ is defined as $\mb{T} = \partial_s \mb{X}$, where $s$ is the unit arc-length parameterization defined by $ds = |\partial_u \mb{X}| du$. Here, $|\mb{x}|$ and 
$\mb{x}^\intercal\mb{y}\equiv\mb{x}\cdot \mb{y}$ denote the Euclidean norm and the inner product of the vectors 
$\mb{x}, \mb{y}\in \mathbb{R}^3$. The curvature $\kappa$ of a curve $\Gamma_t$ is defined as $\kappa = |\partial_{s}\mb{X}\times\partial^2_{s} \mb{X} | = | \partial_{s}^2 \mb{X} |$. If $\kappa > 0$, we can define the Frenet frame along the curve $\Gamma_t$ with unit normal $\mb{N} = {\kappa}^{-1} \partial_{s}^2 \mb{X}$ and binormal vectors $\mb{B} = \mb{T} \times \mb{N}$, respectively. Recall the Frenet-Serret formulae: 
\[
\frac{\text{d}}{\text{d} s} 
\begin{pmatrix}
\mb{T} \\ \mb{N} \\ \mb{B}
\end{pmatrix}
=
\begin{pmatrix}
0 & \kappa & 0 \\
-\kappa & 0 & \tau \\
0 & -\tau & 0
\end{pmatrix}
\begin{pmatrix}
\mb{T} \\ \mb{N} \\ \mb{B}
\end{pmatrix},
\]
where $\tau$ is the torsion of $\Gamma_t$. 

We study a coupled system of evolutionary equations that describes the evolution of closed 3D curves evolving in normal and binormal directions. More specifically, we focus on the motion of a family of curves evolving in 3D and satisfying the geometric law
\begin{equation}
\partial_t\mb{X} = a \partial^2_{s} \mb{X} + \mb{F}(\mb{X}, \partial_{s} \mb{X}) +\alpha\partial_s\mb{X},
\label{eq:ab}
\end{equation}
where $a=a(\mb{X}, \partial_s\mb{X}) > 0$, and $\mb{F}=\mb{F}(\mb{X}, \partial_s\mb{X})$ are bounded and smooth functions of their arguments. Here, $\mb{F}$ is an external force term that restricts the movement of $\Gamma_t$ to a given manifold. The forcing term is discussed further in Theorem \ref{theo-F}. Since $\partial^2_s \mb{X} =\kappa \mb{N}$ and $\mb{B} = \mb{T} \times \mb{N}$ the relationship between the geometric equations (\ref{eq:general}) and (\ref{eq:ab}) is as follows:
\begin{equation}
v_N = a \, \kappa + \mb{F}^\intercal \mb{N}, \quad
v_B = \mb{F}^\intercal  \mb{B}, \quad
v_T = \mb{F}^\intercal  \mb{T} + \alpha.
\label{eq:rel}
\end{equation}
The system of equations (\ref{eq:ab}) is subject to the initial condition
\begin{equation}
\mb{X}(u,0)  = \mb{X_0}(u), u\in I=\mathbb{R}/\mathbb{Z}\simeq S^1, 
\label{init-ab}
\end{equation}
representing parameterization of the initial curve $\Gamma_0$.

\section{Evolution of closed curves on a closed surface without boundary}

In this section, we analyze the evolution of closed curves on a two-dimensional surface without boundary. First, we discuss the evolution of curves in embedded manifolds. 

\subsection{Evolution of curves on embedded manifolds}
Assume ${\mathcal{M}}\in\mathbb{R}^3$ is an embedded manifold given by 
${\mathcal{M}}=\{ \mb{X}=(X_1,X_2,X_3)^\intercal\in\mathbb{R}^3,\ f(\mb{X}) =0\}$ where  $f(\mb{X}):\mathbb{R}^3\to \mathbb{R}$ is a $C^4$ smooth regular map. This means that $\nabla f(\mb{X})\not=0$  for $\mb{X}\in {\mathcal{M}}$.

\begin{figure}
\begin{center}
\includegraphics[width=0.35\textwidth]{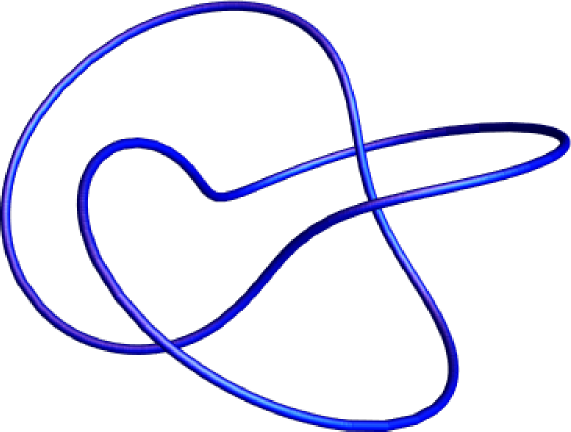}
\quad
\includegraphics[width=0.35\textwidth]{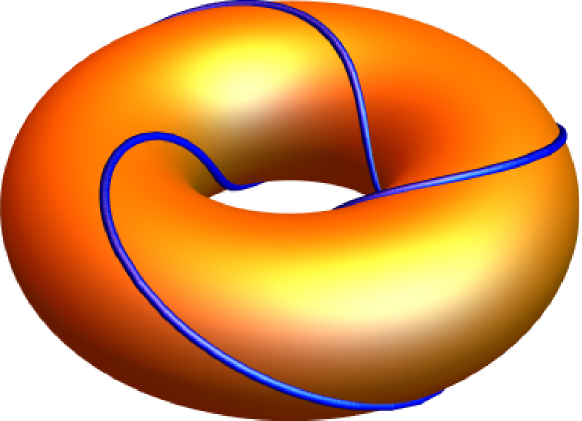}

{\small a) \hskip 5truecm b)}

\caption{A knotted curve a) belonging to the embedded torus surface b). } 

\label{fig-curve-on-torus}
\end{center}
\end{figure}

Let $\phi(u,t) = f(\mb{X}(u,t))$. Then
\begin{eqnarray*}
\partial_t\phi &=& \nabla f(\mb{X})^\intercal\partial_t\mb{X} \\
&=& \nabla f(\mb{X})^\intercal \left( a \partial^2_s\mb{X} + \mb{F} +\alpha \partial_s \mb{X}\right)
= a \nabla f(\mb{X})^\intercal \partial^2_s\mb{X} + \nabla f(\mb{X})^\intercal \mb{F} +\alpha \nabla f(\mb{X})^\intercal \partial_s \mb{X}
\\
&=& a \partial_s \left(\nabla f(\mb{X})^\intercal \partial_s\mb{X}\right) - a \partial_s\mb{X}^\intercal \nabla^2 f(\mb{X}) \partial_s\mb{X} + \nabla f(\mb{X})^\intercal \mb{F} +\alpha \partial_s \phi
\\
&=& a \partial^2_s \phi - a \partial_s\mb{X}^\intercal \nabla^2 f(\mb{X}) \partial_s\mb{X} + \nabla f(\mb{X})^\intercal \mb{F} 
+ \alpha \partial_s \phi
\end{eqnarray*}
because $\partial_s\phi=\nabla f(\mb{X})^\intercal \partial_s\mb{X}$. 

\begin{theorem}\label{theo-F}
Suppose that the external force $\mb{F}$ is given by
\begin{equation}
\mb{F}(\mb{X},\mb{T})= a \frac{\mb{T}^\intercal \nabla^2 f(\mb{X}) \mb{T} + h(f(\mb{X}))}{\vert \nabla f(\mb{X})\vert^2} \nabla f(\mb{X}), 
\label{F}
\end{equation}
$\mb{T}=\partial_s\mb{X}$. Here $h:\mathbb{R}\to\mathbb{R}$ is a smooth function, $h(0)=0$. 
If the initial curve $\Gamma_0 \subset {\mathcal{M}}$ then   the family of curves $\Gamma_t$, for $t>0$, evolving with respect to the geometric equation (\ref{eq:ab})   belongs to the manifold ${\mathcal{M}}=\{ \mb{X}\in\mathbb{R}^3,\ f(\mb{X}) =0\}$. Furthermore, the tangential velocity $v_T=\alpha$.

\end{theorem}

\begin{proof}

If the initial curve $\Gamma_0 \subset {\mathcal{M}}$ then $\phi(\cdot,0) = f(\mb{X}(\cdot,0))=0$. The function $\phi$ is a solution to the parabolic equation:
\begin{equation}
\partial_t\phi  = a \partial^2_s \phi + a h(\phi) + \alpha \partial_s \phi,
\label{parabolic}
\end{equation}
with a zero initial condition $\phi(\cdot,0) =0$. Since $h(0)=0$ we have $\phi(\cdot,t) =0$ for all $t>0$. That is $\Gamma_t \subset {\mathcal{M}}$. 
The projection   $\mb{F}^\intercal\mb{T}$ of $\mb{F}$ to the tangent direction is vanishing. In fact, as $0=\partial_s\phi = \nabla f(\mb{X})^\intercal \mb{T}$, we have 
$\mb{F}(\mb{X},\mb{T})^\intercal\mb{T} = a \frac{\mb{T}^\intercal \nabla^2 f(\mb{X}) \mb{T} + h(f(\mb{X}))}{\vert \nabla f(\mb{X})\vert^2} \nabla f(\mb{X})^\intercal \mb{T} =0$.

\end{proof}

\begin{theorem}\label{theo-lengthshortening}
Suppose that the external force $\mb{F}$ is defined by equation (\ref{F}), and the initial curve $\Gamma_0 \subset {\mathcal{M}}$. Then the geometric flow of the curves, which satisfies the law (\ref{eq:ab}),   is   the length-shortening flow  on the surface $\mathcal{M}$, i.e., $\frac{d}{dt} L(\Gamma_t) \le 0$, where $L(\Gamma_t)$ denotes the total length of the curve $\Gamma_t$.
\end{theorem}

\begin{proof}

Suppose that the initial curve $\Gamma_0 \subset {\mathcal{M}}$. With regard to Theorem~\ref{theo-F} we have $\Gamma_t \subset {\mathcal{M}}$ for any $t\ge 0$. That is, $\phi(s,t)\equiv f(\mb{X}(s,t)) =0$ for all $s\in[0,L(\Gamma_t)]$ and $t\ge 0$.  Hence $0=\partial_s\phi = \nabla f(\mb{X})^\intercal \mb{T}$, and 
\[
0=\partial^2_s\phi = \mb{T}^\intercal \nabla^2 f(\mb{X}) \mb{T} + \nabla f(\mb{X})^\intercal \partial_s \mb{T} = \mb{T}^\intercal \nabla^2 f(\mb{X}) \mb{T} + \kappa \nabla f(\mb{X})^\intercal \mb{N}.
\]
The projection of the velocity $\partial_t\mb{X}$ in the normal direction $\mb{N}$ is given by 
\begin{eqnarray}
v_N &=& (a\,\partial^2_s \mb{X} + \mb{F})^\intercal \mb{N} 
= a\,\kappa + a\,\frac{\mb{T}^\intercal \nabla^2 f(\mb{X}) \mb{T}}{\vert \nabla f(\mb{X})\vert^2} \nabla f(\mb{X})^\intercal \mb{N}
= a\, \kappa  - a\,\kappa \frac{(\nabla f(\mb{X})^\intercal \mb{N})^2}{\vert \nabla f(\mb{X})\vert^2}
\nonumber
\\
&=& a\, \kappa  - a\,\kappa \frac{(\nabla f(\mb{X})^\intercal \mb{N})^2}{\vert \nabla f(\mb{X})\vert^2}
= a\, \kappa \frac{(\nabla f(\mb{X})^\intercal \mb{B})^2}{\vert \nabla f(\mb{X})\vert^2}
\label{normalvelocity}
\end{eqnarray}
because 
$\vert \nabla f(\mb{X})\vert^2 = (\nabla f(\mb{X})^\intercal \mb{N})^2 + (\nabla f(\mb{X})^\intercal \mb{T})^2 + (\nabla f(\mb{X})^\intercal \mb{B})^2 = (\nabla f(\mb{X})^\intercal \mb{N})^2 + (\nabla f(\mb{X})^\intercal \mb{B})^2$. 
\begin{equation}
\frac{d}{dt} L(\Gamma_t) = - \int_{\Gamma_t} \kappa v_N ds 
= - \int_{\Gamma_t} a\, \kappa^2 \frac{(\nabla f(\mb{X})^\intercal \mb{B})^2}{\vert \nabla f(\mb{X})\vert^2} ds \le 0.
\label{dLequation}
\end{equation}

\end{proof}

Any monotonically increasing  modification of the normal velocity $v_N$ given by (\ref{normalvelocity}) again represents a length-shortening flow.
 
\begin{remark}
The geodesic curvature $\kappa^g$ of a curve $\Gamma\subset \mathcal{M}$ can be defined as a projection of the derivative $\partial_s\mb{T}$ to the unit vector $\mb{N}^g$ perpendicular to $\mb{T}$ and belonging to the tangent space $T_X(\mathcal{M})$ at the point $\mb{X}\in\mathcal{M}$. Both vectors belong to the tangent space $T_X(\mathcal{M})$ that is perpendicular to the outer normal vector $\nabla f(\mb{X})$ to the surface $\mathcal{M}$. The unit normal vector to the surface $\mathcal{M}$ is given by the vector $\nabla f(\mb{X})/|\nabla f(\mb{X})|$. Therefore, $\mb{N}^g = (\nabla f(\mb{X})\times \mb{T})/|\nabla f(\mb{X})|$. Hence, 
\begin{eqnarray*}
\kappa^g &=& \partial_s\mb{T}^\intercal \mb{N}^g = \kappa \mb{N}^\intercal \mb{N}^g
= \kappa \mb{N}^\intercal (\nabla f(\mb{X})\times \mb{T})/|\nabla f(\mb{X})|
\\
&=& \kappa  \nabla f(\mb{X})^\intercal (\mb{T}\times \mb{N})/|\nabla f(\mb{X})|
= \kappa  (\nabla f(\mb{X})^\intercal \mb{B})/|\nabla f(\mb{X})| .
\end{eqnarray*}
 
This means that the flow driven by the geodesic curvature with normal velocity $v_{N^g} =  a  \kappa^g$ is the same as the flow of a surface with $v_N$ given by  (\ref{normalvelocity}).
 
According to Mikula and \v{S}ev\v{c}ovi\v{c} \cite[Eq. (12)]{MS3} we have $\frac{d}{dt} L(\Gamma_t) = - \int_{\Gamma_t} \kappa^g v_{N^g} ds$. If the normal velocity $v_{N^g}$ is proportional to the geodesic curvature,  $v_{N^g} =  a  \kappa^g$ then we obtain
\begin{equation}
\frac{d}{dt} L(\Gamma_t) = - \int_{\Gamma_t} \kappa^g v_{N^g} ds 
=  - \int_{\Gamma_t} a (\kappa^g)^2 ds 
= - \int_{\Gamma_t} a\, \kappa^2 \frac{(\nabla f(\mb{X})^\intercal \mb{B})^2}{\vert \nabla f(\mb{X})\vert^2} ds \le 0, 
\label{dLequationgeodesic}
\end{equation}
which is the relation (\ref{dLequation}). 
\end{remark}

\subsection{Evolution of curves on immersed manifolds}

\begin{figure}

\begin{center}
\includegraphics[width=0.3\textwidth]{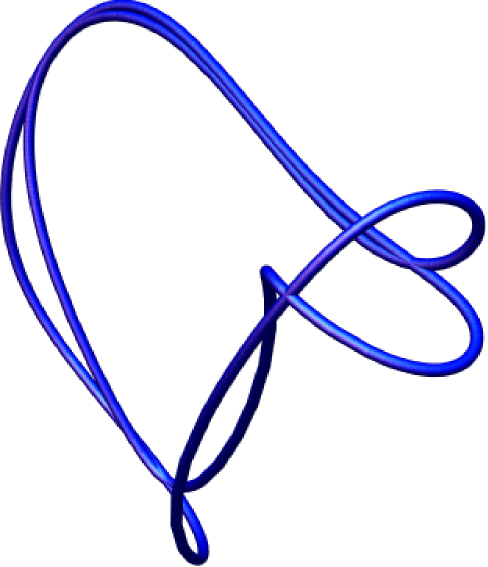}
\quad
\includegraphics[width=0.3\textwidth]{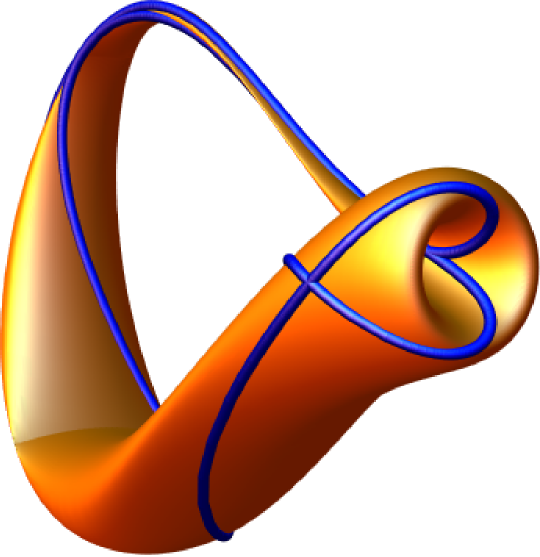}

{\small a) \hskip 5truecm b)}

\caption{a) An initial knotted curve belonging to the immersed  Klein bottle surface b)}

\label{fig-curve-on-klein-bottle}
\end{center}
\end{figure}

In this section, we discuss the evolution of 3D curves on immersed manifolds. We consider an immersed manifold $\mathcal{M}=\{ \mb{X} = \mathcal{X}(\mb{Y}), \quad \mb{Y}\in I \times I\}$ 
where $\mb{Y} = (Y_1, Y_2)^\intercal$ and $\mb{X} = (X_1, X_2, X_3)^\intercal$ are parameterized by immersion $\mathcal{X}:  I\times I  \to \mathbb{R}^3$, where $I=\mathbb{R}/\mathbb{Z}\simeq S^1$ is the periodic interval.

\begin{proposition}\label{3D2D}
Suppose that the  function $\mb{Y}(\cdot, t) \subset I\times I$ is a 1-periodic solution to the parabolic equation: 
\begin{equation}
\partial_t \mb{Y} =a \partial_s^2 \mb{Y} + \mb{G}(\mb{Y}, \partial_s\mb{Y}) +\alpha \partial_s\mb{Y},\quad \mb{Y}(\cdot,0) =\mb{Y}_0(\cdot),
\label{Yimmersion}
\end{equation}
where $ds =  |\nabla \mathcal{X}(\mb{Y})^\intercal 
\partial_u \mb Y| du$, $a=a(\mathcal{X}(\mb{Y}), \nabla \mathcal{X}(\mb{Y})^\intercal \partial_{s} \mb{Y})$, and
  $\mb{G}: \mathbb{R}^2\times  \mathbb{R}^2\to \mathbb{R}^2$ is a $C^4$ smooth function. 
 
Then the closed curve $\Gamma_t=\{ \mb{X}(s,t), s\in[0,L(\Gamma_t)]\}\subset\mathcal{M}$, $\mb{X}(\cdot, t) = \mathcal{X}(\mb{Y}(\cdot,t))$ evolves according to the second order parabolic geometric equation:
\begin{equation}
\partial_t \mb{X} = a \partial_s^2 \mb{X}  + \mb{F}(\mb{X}, \partial_s\mb{X}) + \alpha\partial_s\mb{X}, \quad  \mb{X}(\cdot,0) =\mb{X}_0(\cdot),
\label{Ximmersion}
\end{equation}
where $\mb{F}(\mb{X}, \partial_s\mb{X}) = \nabla \mathcal{X}(\mb{Y})^\intercal \mb{G}(\mb{Y}, \partial_s\mb{Y}) - a  \partial_{s} \mb{Y}^\intercal \nabla^2 \mathcal{X}(\mb{Y}) \partial_{s} \mb{Y} $. Here $\mb{Y}\in I\times I$ is such that $\mb{X}= \mathcal{X}(\mb{Y})\in\mathcal{M}$ and $\partial_s\mb{Y} = (\nabla \mathcal{X}(\mb{Y}) \nabla \mathcal{X}(\mb{Y})^\intercal)^{-1} \nabla \mathcal{X}(\mb{Y}) \partial_s\mb{X}$ where the $2\times 3$ matrix 
\\
$(\nabla \mathcal{X}(\mb{Y}) \nabla \mathcal{X}(\mb{Y})^\intercal)^{-1} \nabla \mathcal{X}(\mb{Y})$ is the left Moore-Penrose pseudoinversion of $3\times 2$ matrix  $\nabla \mathcal{X}(\mb{Y})^\intercal$.

The curve $\Gamma_t, t\ge0$, evolves according to the geometric equation (\ref{eq:general}) with the normal, binormal and tangential velocities given by (\ref{eq:rel}).

\end{proposition}

\begin{proof}

Sin`ce $\mb{X} = \mathcal{X}(\mb{Y})$ we have
$\partial_{t} \mb{X} = \nabla \mathcal{X}(\mb{Y})^\intercal \partial_{t} \mb{Y}$, 
\quad $\partial_{s} \mb{X} = \nabla \mathcal{X}(\mb{Y})^\intercal \partial_{s} \mb{Y}$,  
\quad $\partial_{s}^2 \mb{X} = \nabla \mathcal{X}(\mb{Y})^\intercal \partial_{s}^2 \mb{Y} + \partial_{s} \mb{Y}^\intercal \nabla^2 \mathcal{X}(\mb{Y}) \partial_{s} \mb{Y}$, 
where $\nabla \mathcal{X}=\nabla \mathcal{X}(\mb{Y})$ is the $2\times 3$ matrix:
\[
\nabla \mathcal{X} =
\setlength\arraycolsep{2pt}
    \begin{pmatrix} 
    \frac{\partial \mathcal{X}_1}{\partial Y_1} & \frac{\partial \mathcal{X}_2}{\partial Y_1} &\frac{\partial \mathcal{X}_3}{\partial Y_1} 
    \\[5pt]
    \frac{\partial \mathcal{X}_1}{\partial Y_2} & \frac{\partial \mathcal{X}_2}{\partial Y_2} &\frac{\partial \mathcal{X}_3}{\partial Y_2} 
    \end{pmatrix}.
\]
The vector $\partial_{s} \mb{Y}^\intercal \nabla^2 \mathcal{X}(\mb{Y}) \partial_{s} \mb{Y}\in \mathbb{R}^3$ is constructed as follows:
\[
\partial_{s} \mb{Y}^\intercal \nabla^2 \mathcal{X}(\mb{Y}) \partial_{s} \mb{Y}
=
(\partial_{s} \mb{Y}^\intercal \nabla^2 \mathcal{X}_k(\mb{Y}) \partial_{s} \mb{Y})_{k=1,2,3} \in \mathbb{R}^3,
\quad\nabla^2 \mathcal{X}_k = 
\left( \frac{\partial^2 \mathcal{X}_k}
{\partial Y_i\partial Y_j}\right)_{i,j=1,2} .
\]
The $2\times 3$ matrix $\nabla \mathcal{X}(\mb{Y})$ has full rank  because the mapping $\mathcal{X}$ is assumed to be an immersion. Furthermore, as $\mb{X} = \mathcal{X}(\mb{Y})$ we have $ds = |\partial_u \mb{X}| du = | \nabla \mathcal{X}(\mb{Y})^\intercal \partial_{u} \mb{Y}| du$. Therefore, the derivative of $\mb{Y}$   with respect to the arc-length parameterization $s$ of the curve $\mb{X}$ can be written as 
\[
\frac{\partial \mb{Y}}{\partial s} = \frac{1}{|\nabla \mathcal{X}(\mb{Y})^\intercal \partial_u \mb Y|}  \frac{\partial \mb{Y}}{\partial u}.
\]
Then
\begin{eqnarray*}
\partial_{t} \mb{X} &=& \nabla \mathcal{X}(\mb{Y})^\intercal \partial_{t} \mb{Y}
= \nabla \mathcal{X}(\mb{Y})^\intercal
\left( a \partial_s^2 \mb{Y} + \mb{G}(\mb{Y}, \partial_s\mb{Y}) +\alpha \partial_s\mb{Y}
\right)
\\
&=&  a \partial_s^2 \mb{X}  +\alpha\partial_s\mb{X}
- a \partial_{s} \mb{Y}^\intercal \nabla^2 \mathcal{X}(\mb{Y}) \partial_{s} \mb{Y} 
+ \nabla \mathcal{X}(\mb{Y})^\intercal \mb{G}(\mb{Y}, \partial_s\mb{Y})
\\
&=& a \partial_s^2 \mb{X}  +\alpha\partial_s\mb{X} + \mb{F}(\mb{X}, \partial_s\mb{X}),
\end{eqnarray*}
as claimed.
\end{proof}

As an example of a non-orientable  immersed manifold in $\mathbb{R}^3$, we can consider the Klein bottle   surface   parameterized by $\mathcal{X}= (\mathcal{X}_1,  \mathcal{X}_2,  \mathcal{X}_3)$
 
where
\begin{eqnarray*}
\mathcal{X}_1(u,v) &=& -(2/15) \cos(2\pi u) (3 \cos(2\pi v) - 30 \sin(2\pi u) + 90 \cos(2\pi u)^4 \sin(2\pi u)
\\
&&-  60 \cos(2\pi u)^6 \sin(2\pi u) + 5 \cos(2\pi u) \cos(2\pi v) \sin(2\pi u)),
\\
\mathcal{X}_2(u,v) &=& -(1/15) \sin(2\pi u) (3 \cos(2\pi v) - 3 \cos(2\pi u)^2 \cos(2\pi v) - 48 \cos(2\pi u)^4 \cos(2\pi v)  
\\
&& + 48 \cos(2\pi u)^6 \cos(2\pi v) + 60 \sin(2\pi u) + 5 \cos(2\pi u) \cos(2\pi v) \sin(2\pi u) 
\\
&& -  5 \cos(2\pi u)^3 \cos(2\pi v) \sin(2\pi u) - 80  \cos(2\pi u)^5 \cos(2\pi v) \sin(2\pi u) 
\\
&& + 80  \cos(2\pi u)^7 \cos(2\pi v) \sin(2\pi u)  ),
\\
\mathcal{X}_3(u,v) &=& (2/15) (3 + 5 \cos(2\pi u) \sin(2\pi u)) \sin(2\pi v) .
\end{eqnarray*}
The surface of the Klein bottle is shown in Fig.~\ref{fig-curve-on-klein-bottle}, b). The initial curve $\mb{X}_0$ is parameterized by 
\[
\mb{X}_0(u) = \mathcal{X}(k u, l u), \quad u\in I,
\]
where $k=1, l=4$. It is shown in Fig.~\ref{fig-curve-on-klein-bottle}, a). The Klein bottle is an immersed manifold in $\mathbb{R}^3$. It is well known that it can be embedded in $\mathbb{R}^4$ but it cannot be embedded in $\mathbb{R}^3$.

\begin{remark}

The $2\times 2$ matrix $\nabla \mathcal{X}(\mb{Y}) \nabla \mathcal{X}(\mb{Y})^\intercal$ is positive definite because the mapping $ \mathcal{X}$ is immersion. For the torus surface, we have 
\[
det (\nabla \mathcal{X}(\mb{Y}) \nabla \mathcal{X}(\mb{Y})^\intercal) = 16\pi^4 r^2 (R+r\cos(2\pi v))^2  \ge  16\pi^4 r^2 (R-r)^2 > 0.
\]
On the other hand, for the Klein bottle surface shown in Fig.~\ref{fig-curve-on-klein-bottle}, b) we observe large spectral variations in the $2\times 2$ positive definite matrix $\nabla \mathcal{X}(\mb{Y}) \nabla \mathcal{X}(\mb{Y})^\intercal$. That is, 
\[
det (\nabla \mathcal{X}(\mb{Y}) \nabla \mathcal{X}(\mb{Y})^\intercal) \in (0.0145, 32020).
\]
\end{remark}

\section{Existence and uniqueness of classical H\"older smooth solutions}

In this section, we present   theoretical  results on the existence and uniqueness of the classical H\"older smooth solution to the system of equations (\ref{eq:ab}) for the motion of the time-dependent family of curves $\Gamma_t=\{ \mb{X}(u, t),\ u\in I\},\ t\ge 0$, evolving in $\mathbb{R}^3$. Furthermore, we prove the existence and uniqueness of solutions $\mb{Y}=\mb{Y}(u,t)$ of the non-linear equation (\ref{Yimmersion}) (see Proposition~\ref{3D2D}). We employ the analytical framework developed by Bene\v{s}, Kol{\' a}{\v r}, \v{S}ev\v{c}ovi\v{c} \cite{kolar2022} in the context of curve evolution in $\mathbb{R}^3$ and Mikula and \v{S}ev\v{c}ovi\v{c} \cite{sevcovic2001evolution, MS2004, MMAS2004, MS3} for the evolution of planar curves. The proof of the existence and uniqueness of solutions in  H\"older spaces is based on the abstract theory of analytic semiflows in Banach spaces, as established by DaPrato and Grisvard \cite{daprato}, Angenent \cite{Angenent1990b, Angenent1990} and Lunardi \cite{Lunardi1984}. 
In the proof technique, the position vector equation (\ref{eq:ab}) with uniform tangential velocity $v_T$ is analyzed. 

The nonlinear parabolic equation (\ref{eq:ab}), and similarly (\ref{Yimmersion}), can be rewritten as the abstract parabolic equation: $\partial_t \mb{X} = \mathscr{F}(\mb{X}), \ \mb{X}(0) = \mb{X}_0$, on a scale of suitable Banach spaces. Suppose $0<\varepsilon<1$, and $k\in \mb{N}$. The little H\"older space $h^{k+\varepsilon}(S^1)$ is the Banach space defined as the closure of $C^\infty$ smooth functions defined in the periodic domain $S^1$. The norm is the sum of the $C^k$ norm and the $\varepsilon$-H\"older semi-norm of the $k$-th derivative. Next, we define the following scale of Banach spaces consisting of $(2k+\varepsilon)$-H\"older continuous functions in the periodic domain $I\simeq S^1$:
\begin{equation}
\mc{E}^X_k = h^{2k +\varepsilon}(S^1)\times h^{2k +\varepsilon}(S^1) \times h^{2k +\varepsilon}(S^1), \quad \mc{E}^Y_k = h^{2k +\varepsilon}(S^1)\times h^{2k +\varepsilon}(S^1), \quad k=0, \frac12, 1 .
\label{Espaces}
\end{equation}
For the application of the theory of nonlinear analytic semiflovs due to DaPrato and Grisvard \cite{daprato}, Angenent \cite{Angenent1990b, Angenent1990}, and Lunardi \cite{Lunardi1984}, it is sufficient to prove that, for any $\tilde{\mb{X}}$ linearization $\mathscr{A}=\mathscr{F}'(\tilde{\mb{X}} )$  generates an analytic semigroup in the space in $\mc{E}^Z_0$, and it belongs to the maximal regularity class between Banach spaces $\mc{E}^Z_1$ and $\mc{E}^Z_0$ for $Z\in\{X,Y\}$. 
Now we can state the following result, stating the local existence and uniqueness of solutions to the system of nonlinear geometric equations (\ref{eq:ab}). 

\begin{theorem}\label{theo-main}
Assume ${\mathcal{M}}\in\mathbb{R}^3$ is an embedded manifold given by 
${\mathcal{M}}=\{ \mb{X} \in\mathbb{R}^3,\ f(\mb{X}) =0\}$ where  $f:\mathbb{R}^3\to \mathbb{R}$ is a $C^4$ smooth regular map, i.e., $\nabla f(\mb{X})\not=0$  for $\mb{X}\in {\mathcal{M}}$, and $a=a(\mb{X}, \mb{T}) > 0$ is a $C^4$ smooth function, $a:\mathbb{R}^2 \times \mathbb{R}^2 \to \mathbb{R}$.  Assume that $h:\mathbb{R}\to\mathbb{R}$ is a $C^4$ smooth function, $h(0)=0$. Suppose that the parameterization $\mb{X}_0$ of the initial curve $\Gamma_0$ belongs to the H\"older space $\mc{E}^X_1$, and is uniformly parameterized $|\partial_u\mb{X}_0(u)| = L(\Gamma_0) >0$ for all $u\in I$. Assume that the total tangential velocity $v_T$ preserves the relative local length. Then there exists $T>0$ and the unique solution $\mb{X}$ to the initial value problem:
\[
\partial_t\mb{X} = a \partial^2_{s} \mb{X} + \mb{F}(\mb{X}, \partial_{s} \mb{X}) +\alpha\partial_s\mb{X},
\quad \mb{X}(\cdot, 0)= \mb{X}_0(\cdot), \ u\in I,\ t\in[0,T),
\]
where the external force $\mb{F}$ is given by (\ref{F}). Moreover, 
$\mb{X}\in  C([0,T], \mc{E}^X_1) \cap C^1([0,T], \mc{E}^X_0)$.

\end{theorem}

\begin{proof}
We rewrite the non-linear parabolic equation (\ref{eq:ab}) in the form: $\partial_t \mb{X} = \mathscr{F}(\mb{X})$, where $ \mathscr{F}(\mb{X})=a \partial^2_{s} \mb{X} + \mb{F}(\mb{X}, \partial_{s} \mb{X}) +\alpha\partial_s\mb{X}$. Under the assumptions made on functions $a, f$, and $h$, the mapping 
\[
\mc{E}^X_{\frac12} \ni \mb{X} \mapsto \mb{F}(\mb{X}, \partial_{s} \mb{X}) +\alpha\partial_s\mb{X} \in \mc{E}^X_0
\]
is $C^1$ mapping from the Banach space $\mc{E}^X_{\frac12}$  to $\mc{E}^X_0$. 

Assume $\tilde{\mb{X}}$ belongs to the H\"older space $\mc{E}^X_1$, and is uniformly parameterized, $|\partial_u\tilde{\mb{X}}| = L(\tilde{\Gamma}) >0$ for all $u\in I$. Then the linearization $\mathscr{A}=\mathscr{F}'(\tilde{\mb{X}} )$ can be decomposed as follows: $\mathscr{A} = \mathscr{A}_0 + \mathscr{A}_1$, where the principal part $\mathscr{A}_0$ containing the second order derivative has the form
$\mathscr{A}_0 \mb{X} = \frac{\tilde{a}}{L(\tilde{\Gamma})^2} \partial^2_u\mb{X}$, 
where $\tilde{a}={a}(\tilde{\mb{X}}, \tilde{\mb{T}})$,  $\tilde{\mb{T}}=\partial_s \tilde{\mb{X}}$. It is known that $\mathscr{A}_0$ generates an analytic semigroup in space $\mc{E}^X_0$, and it belongs to the maximal regularity class between Banach spaces $\mc{E}^X_1$ and $\mc{E}^X_0$ (cf.  \cite{Angenent1990b, Angenent1990}). Furthermore, the operator $\mathscr{A}_1 = \mathscr{A} - \mathscr{A}_0$ contains the first order derivative of $\mb{X}$. It is a bounded linear operator $\mathscr{A}_1: \mathscr{E}_{\frac12}\to \mathscr{E}_0$. Therefore, the operator $\mathscr{A}_1$ considered as a mapping from $\mathscr{E}_1\to \mathscr{E}_0$ has the relative zero norm with respect to $\mathscr{A}_0$ . Therefore, linearization $\mathscr{A}$ belongs to the maximal regularity class ${\mathcal M}(\mathscr{E}_1,\mathscr{E}_0)$ because this class is closed with respect to perturbation with the relative zero norm (cf.  \cite[Lemma 2.5]{Angenent1990}, DaPrato and Grisvard \cite{daprato}, Lunardi \cite{Lunardi1984}). The proof now follows \cite[Theorem 2.7]{Angenent1990} due to Angenent.
\end{proof}

\begin{theorem}\label{theo-main-Y}
Assume ${\mathcal{M}}$ is a manifold immersed in $\mathbb{R}^3$ given by ${\mathcal{M}}= \{ \mb{X} = \mathcal{X}(\mb{Y}), \quad \mb{Y}\in I \times I\}$ where $\mathcal{X}:I \times I\to \mathbb{R}^3$, is a $C^4$ smooth immersion, $rank(\nabla\mathcal{X}(\mb{Y})) = 2$ for all $\mb{Y}\in I \times I$, and $a=a(\mb{X}, \mb{T}) > 0$ is a   $C^4$   smooth function, $a:\mathbb{R}^3 \times \mathbb{R}^3 \to \mathbb{R}$. Suppose that the initial condition $\mb{Y}_0$ belongs to the H\"older space $\mc{E}^Y_1$, and $|\partial_u\mb{Y}_0(u)| >0$ for all $u\in I$.  Then there exists $T>0$ and the unique solution $\mb{Y}\in  C([0,T], \mc{E}^Y_1) \cap C^1([0,T], \mc{E}^Y_0)$  to the initial value problem:
\[
\partial_t \mb{Y} = a \partial_s^2 \mb{Y} + \mb{G}(\mb{Y}, \partial_s\mb{Y}) +\alpha \partial_s\mb{Y},\quad \mb{Y}(\cdot,0) =\mb{Y}_0(\cdot), \ u\in I,\ t\in[0,T),
\]
where $ds =  |\nabla \mathcal{X}(\mb{Y})^\intercal 
\partial_u \mb Y| du$, $a=a(\mathcal{X}(\mb{Y}), \nabla \mathcal{X}(\mb{Y})^\intercal \partial_{s} \mb{Y})$ and 
  $\mb{G}: \mathbb{R}^2\times  \mathbb{R}^2\to \mathbb{R}^2$ is a $C^4$ smooth function. 
 
\end{theorem}

\begin{proof}
Based on the assumptions made on the immersion mapping $\mathcal{X}:I \times I\to \mathbb{R}^3$, we have the following:
\[
\mc{E}^Y_{\frac12} \ni \mb{Y} \mapsto \mb{G}(\mb{Y}, \partial_{s} \mb{Y}) +\alpha\partial_s\mb{Y} \in \mc{E}^Y_0
\]
is a $C^1$ mapping from the Banach space $\mc{E}^Y_{\frac12}$  to $\mc{E}^Y_0$. 
The rest of the proof is essentially the same as that of Theorem~\ref{theo-main}. 
\end{proof}
 
In general, we cannot take $T = +\infty$. Indeed, a closed curve on a torus or plane surface may shrink to a point in finite time.

\section{Flowing finite volumes numerical discretization scheme}

In this section, we present a numerical discretization scheme for solving the system of equations (\ref{eq:ab}) with tangential velocity $\alpha$. The discretization utilizes the method of lines with spatial discretization achieved through the finite-volume method. We focus on the evolution of the curves $\Gamma_t, t\ge 0$, satisfying the governing equation: 
\begin{equation}
    \label{eq:num_start}
    \partial_t \mb{X} = a \partial_s^2 \mb{X}
    + \mb{F}(\mb{X}, \partial_s\mb{X})
    + \alpha \mb{T}.
\end{equation}
We consider $M$ discrete nodes $\mb{X}_k = \mb{X}(u_k)$,  $k = 0,1,2, \ldots M$, $u_k = kh \in [0,1]$, where $h=1/M$, $\mb{X}_0 = \mb{X}_M$ along the curve $\Gamma_t$. The dual nodes are defined as $\mb{X}_{k \pm \frac12} = \mb{X}(u_k \pm h/2)$ (see Fig. \ref{FVMfig}) and $(\mb{X}_k + \mb{X}_{k+1}) / 2$ is the midpoint of the line segment connecting nodes $\mb{X}_k$ and $\mb{X_{k+1}}$. This midpoint differs from $\mb{X}_{k \pm \frac12} \in \Gamma_t$. The $k$-th segment $\mathcal{S}_k$ of $\Gamma_t$ between the nodes $\mb{X}_{k + \frac12}$ and $\mb{X}_{k - \frac12}$ represents the finite volume. Integration of equation (\ref{eq:num_start}) over such a  volume yields
\begin{equation}
\label{eq:num_int}
\begin{split}
\int_{u_{k-\frac12}}^{u_{k+\frac12}} \partial_t \mb{X} |\partial_u \mb{X}| d u = &
\int_{u_{k-\frac12}}^{u_{k+\frac12}} a\frac{\partial}{\partial_u} \left( \frac{\partial_u \mb{X}}{|\partial_u \mb{X}|} \right) d u
+  \int_{u_{k-\frac12}}^{u_{k+\frac12}} \mb{F} |\partial_u \mb{X}| d u
+ \int_{u_{k-\frac12}}^{u_{k+\frac12}} \alpha \partial_u \mb{X} d u.
\end{split}
\end{equation}

\begin{figure}
\begin{center}
\includegraphics[width=0.35\textwidth]{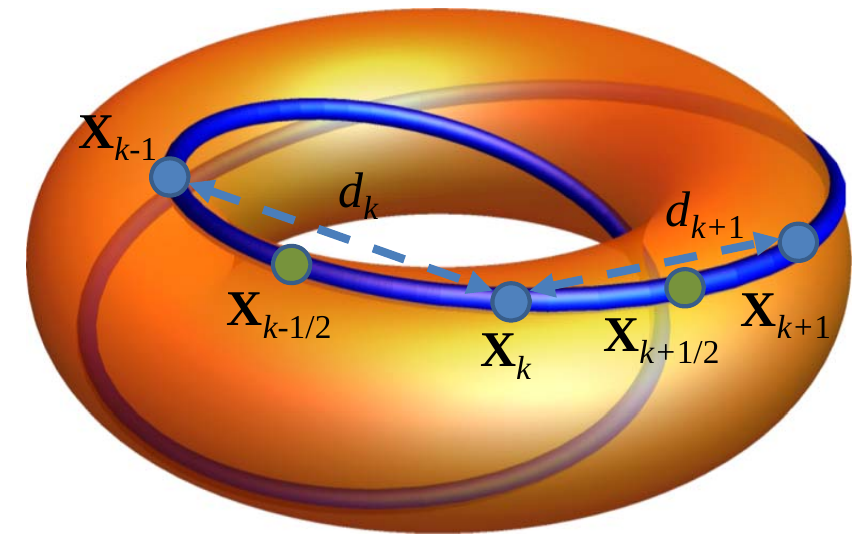}
\end{center}
\caption{Discretization of a segment of a 3D curve on a surface by the method of flowing finite volumes.
}
\label{FVMfig}
\end{figure}

To discretize the governing system of equations, we assume that $\partial_t \mb{X}, \partial_u \mb{X}, \mb{F}, \alpha, \kappa, a, b, \mb{T}$, and $\mb{N}$ remain constant on the finite volume $\mathcal{S}_k$ bounded by the nodes $\mb{X}_{k - \frac12}$ and $\mb{X}_{k + \frac12}$. These variables assume the values $\partial_t \mb{X}_k, \partial_u \mb{X}_k, \mb{F}_k, \alpha_k, \kappa_k, \mb{T_k}$, and $\mb{N_k}$, respectively. In the approximation of the nonlocal vector function $\mb{F}_k$, the curve $\Gamma_t$ used to define $\mb{F}$ is replaced by a polygonal curve having vertices at $(\mb{X}_0, \mb{X}_1, \ldots, \mb{X}_M)$. The numerical approximation of the tangential velocity $\alpha_k$ is summarized at the end of this section.

Let us denote $d_k = |\mb{X}_k - \mb{X}_{k-1}|$ for $k=1,2,\ldots,M,M+1$, where $\mb{X}_M = \mb{X}_0$ and $\mb{X}_{M+1} = \mb{X}_{1}$ for the closed curve $\Gamma$ and we approximate the integral expressions in (\ref{eq:num_int}) by means of the flowing finite volume method as follows:
\begin{equation*}
\begin{split}
& \int_{u_{k-\frac12}}^{u_{k+\frac12}} \partial_t \mb{X} |\partial_u \mb{X}| d  u
 \approx \frac{d  \mb{X}_k}{d t} \frac{d_{k+1} + d_k}{2},
\quad 
\int_{u_{k-\frac12}}^{u_{k+\frac12}}
a \partial_u \frac{\partial_u \mb{X}}{|\partial_u \mb{X}|} d u
\approx
a_k \left(\frac{\mb{X}_{k+1} - \mb{X}_k}{d_{k+1}} - \frac{\mb{X}_k - \mb{X}_{k-1}}{d_k}\right),
\\
&\int_{u_{k-\frac12}}^{u_{k+\frac12}} \mb{F} |\partial_u \mb{X}| d u
\approx  \mb{F}_k \frac{d_{k+1} + d_{k}}{2},
\qquad
\int_{u_{k-\frac12}}^{u_{k+\frac12}} \alpha \partial_u \mb{X} d u\approx  \alpha_k \frac{\mb{X}_{k+1} - \mb{X}_{k-1}}{2}.
\end{split}
\end{equation*}
Here $a_k=a(\mb{X}_k, \mb{T}_k)$ is the diffusion coefficient evaluated at $(\mb{X}_k, \mb{T}_k)$.    The estimation of the nonnegative curvature $\kappa$ along with the tangent vector $\mb{T}$ and the normal vector $\mb{N}$, where $\kappa \mb{N} = \partial_s \mb{T}$, can be expressed as follows:
\begin{equation*}
\begin{split}
\mb{N}_k & \approx \frac{1}{\delta + \kappa_k} \frac{2}{d_k + d_{k+1}}
\left(\frac{\mb{X}_{k+1} - \mb{X}_k}{d_{k+1}} - \frac{\mb{X}_k - \mb{X}_{k-1}}{d_k}\right), 
\qquad
\mb{T}_k  \approx \frac{\mb{X}_{k+1} - \mb{X}_{k-1}}{d_{k+1}+d_k},
\\
\kappa_k & \approx  
\left|\frac{2}{d_k + d_{k+1}}
\left(\frac{\mb{X}_{k+1} - \mb{X}_k}{d_{k+1}} - \frac{\mb{X}_k - \mb{X}_{k-1}}{d_k}\right)
\right|,
\end{split}
\end{equation*}
where $0<\delta\ll 1$ is a small regularization parameter.

The semidiscrete scheme for resolving (\ref{eq:num_start}) is expressed as follows.
\begin{eqnarray}
\label{ODEs}
\frac{d  \mb{X}_k}{d t}
&=&
a_k  \frac{2}{d_{k+1}+d_k}\left(\frac{\mb{X}_{k+1} - \mb{X}_k}{d_{k+1}} - \frac{\mb{X}_k - \mb{X}_{k-1}}{d_k}\right)
  + \mb{F}_k 
 + \alpha_k \frac{\mb{X}_{k+1} - \mb{X}_{k-1}}{d_{k+1} + d_{k}},
 \nonumber
\\
\mb{X}_k(0) &=& \mb{X}_{ini}(u_k), \quad \text{for} \ k = 1, \ldots, M.
\end{eqnarray}

Recall that the tangential component $\alpha$ of the velocity vector for evolving closed curves $\{\Gamma_t, t\ge 0\}$ maintains their shape unchanged (cf. Epstein \cite{ASENS_1990_4_23_2_229_0}). However, for a numerical solution of (\ref{eq:num_start}), the careful selection of the tangential velocity functional $\alpha$ is key to preserve the stability of the computational procedure (cf. Mikula and \v{S}ev\v{c}ovi\v{c} \cite{sevcovic2001evolution, MS2004, MMAS2004}). The significance of tangential velocity is also substantial in theoretical analyzes of curve evolution, as shown by the works of Hou et al. \cite{Hou}, Kimura \cite{Kimura}.
  
Barrett et al.\cite{Barret2010, Barret2012} and Elliott and Fritz \cite{Elliot2017}, explored gradient and elastic flows for closed and open curves in $\mathbb{R}^d$, where $d\ge 2$, and formulated a numerical approximation method to effectively redistribute the tangential component. Furthermore, the relevance of tangential velocity is recognized in material science investigations by Bene\v{s}, Kol\'a\v{r}, and \v{S}ev\v{c}ovi\v{c} \cite{BKS2017} and in the context of interactive evolving curves \cite{BKS2020}. In a different field, Garcke, Kohsaka, and \v{S}ev\v{c}ovi\v{c} \cite{Garcke2009} applied uniform tangential redistribution to theoretically confirm the nonlinear stability of curvature-induced flows with triple junctions in planes. Reme\v{s}\'{\i}kov\'a \emph{et al.} \cite{MS2014} analyzed the tangential redistribution effects for the flows of closed manifolds in $\mathbb{R}^n$. 
 
Calculating the time change of the ratio of the relative local length $|\partial_u\mb{X}(u,t)|$ and the total curve length $L(\Gamma_t) = \int_0^1 |\partial_u\mb{X}(u,t)| du$, it is possible to derive an equation for the unknown $\alpha$ (cf. \cite{vsevvcovivc2011evolution}):
 
\begin{equation}
\frac{\partial }{\partial t}
\frac{|\partial_u\mb{X}(u,t)|}{L(\Gamma_t)} = 
\frac{|\partial_u\mb{X}(u,t)|}{L(\Gamma_t)}
\left(
\partial_{s} \alpha  -\kappa v_N +  \langle  \kappa v_N\rangle \right), 
\quad\text{where}\ \langle  \kappa v_N\rangle = \frac{1}{L(\Gamma_t)} \int_{\Gamma_t} \kappa v_N ds.
\label{alpha}
\end{equation}
The meaning of $\langle \cdot \rangle$ is the average value of a scalar quantity along the curve $\Gamma_t$. Therefore, the ratio $|\partial_u\mb{X}(u,t)|/L(\Gamma_t)$ is constant with respect to the time $t$, i.e.
\begin{equation}
\frac{|\partial_u\mb{X}(u,t)|}{L(\Gamma_t)} = \frac{|\partial_u\mb{X}(u,0)|}{L(\Gamma_0)},
\quad \text{for any} \ t\ge0, 
\label{uniformalpha}
\end{equation}
provided that the tangential velocity satisfies $\partial_{s} \alpha = \kappa v_N -  \langle  \kappa v_N\rangle$. Another suitable choice of the tangential velocity $\alpha$ is the so-called asymptotically uniform tangential velocity proposed and analyzed by Mikula and \v{S}ev\v{c}ovi\v{c} in \cite{MS2004, MMAS2004}. 
 
Taking into account a positive damping parameter $\omega$ and tangential velocity $\alpha$ as the solution of the following equation
 
\begin{equation}
\partial_{s} \alpha = \kappa v_N  - 
 \langle  \kappa v_N\rangle  + \left( \frac{L(\Gamma_t)}{|\partial_u\mb{X}(u,t)|} - 1\right) \omega ,
\label{alpha-asymptotic}
\end{equation}
then, using (\ref{alpha}) we obtain
$\lim_{t\to \infty} \frac{|\partial_u\mb{X}(u,t)|}{L(\Gamma_t)} =1 $ uniformly with respect to $u\in[0,1]$ provided that $\omega>0$ is a positive constant. This means that the redistribution becomes asymptotically uniform. 
The numerical approximation of the tangential velocity (\ref{alpha-asymptotic}) follows from \cite{vsevvcovivc2011evolution}. It requires discrete values of curvature $\kappa_k$, normal velocity $v_{N,k}$, and segment lengths $d_k$, and a straightforward trapezoidal integration is used
(cf. \cite{vsevvcovivc2011evolution}).  The values $\alpha_0 = \alpha_M$ are chosen so that $\sum_{i=1}^M \alpha_i (d_{i+1}+d_i)/2 = 0$.
Then the values $\alpha_k$ for $k=0,1,\ldots,M$ are uniquely given and the direct integration of (\ref{alpha-asymptotic}) leads to the following formulae
\begin{equation}
\begin{split}
\label{eq:num_alpha}
\alpha_i &= \alpha_1 
+ \sum_{k=2}^i \left[ \kappa_i v_{N,i} d_i  - 
 \langle  \kappa v_{N}\rangle d_i + \left( \frac{L(\Gamma_t)}{M} - d_i\right) \omega 
 \right], \quad i = 2,\ldots,M, \\
 \alpha_1 & = - \frac{1}{\sum_{i=1}^M \frac{(d_{i+1} + d_i)}{2}} 
 \left\{
 \sum_{i=2}^M \frac{(d_{i+1} + d_i)}{2} 
 \left(
 \sum_{k=2}^i \left[ \kappa_i v_{N,i} d_i  - 
 \langle  \kappa v_{N}\rangle d_i + \left( \frac{L(\Gamma_t)}{M} - d_i\right) \omega 
 \right]
 \right)
 \right\}.
\end{split}    
\end{equation}

\begin{remark} 
We present a finite difference numerical approximation of the parabolic equation (\ref{Yimmersion}) that governs the flow on an immersed manifold. The second derivative $\partial_s^2 \mb{Y}_k$ appearing in (\ref{Yimmersion}) is approximated as follows: 
 
\begin{eqnarray*}
&&
\frac{|\mb{Y}_{k+1} - \mb{Y}_k|+|\mb{Y}_{k} - \mb{Y}_{k-1}|}{2} \partial_s^2 \mb{Y}_k\approx
\int_{u_{k-\frac12}}^{u_{k+\frac12}}
\frac{1}{| \nabla \mathcal{X}(\mb{Y})^\intercal \mb{t}(\mb{Y})|}
\frac{\partial}{\partial u} \left( \frac{1}{| \nabla \mathcal{X}(\mb{Y})^\intercal \mb{t}(\mb{Y})|}\frac{\partial_u \mb{Y}}{|\partial_u \mb{Y}|}\right) d u
\\
&&\approx
\frac{1}{| \nabla \mathcal{X}(\mb{Y}_k)^\intercal \mb{t}(\mb{Y}_k)|}\left(
   \frac{1}{| \nabla \mathcal{X}(\mb{Y}_{k+1})^\intercal \mb{t}(\mb{Y}_{k+1})|} 
 \frac{\mb{Y}_{k+1} - \mb{Y}_k}{|\mb{Y}_{k+1} - \mb{Y}_k|} - 
  \frac{1}{| \nabla \mathcal{X}(\mb{Y}_k)^\intercal \mb{t}(\mb{Y}_k)|}
 \frac{\mb{Y}_k - \mb{Y}_{k-1}}{|\mb{Y}_k - \mb{Y}_{k-1}|}\right)
\end{eqnarray*}
where $\mb{t}(\mb{Y}_k)\approx (\mb{Y}_k - \mb{Y}_{k-1})/|\mb{Y}_k - \mb{Y}_{k-1}|$.
Similarly,  the first derivatives   $\partial_s\mb{Y}$ and $\partial_t\mb{Y}$ in  (\ref{Yimmersion})   are approximated by means of finite differences as follows: 
\[
 \partial_s \mb{Y}_k = \frac{1}{| \nabla \mathcal{X}(\mb{Y}_k)^\intercal \mb{t}(\mb{Y}_k)|}  \partial_r \mb{Y}_k  \approx
\frac{1}{| \nabla \mathcal{X}(\mb{Y}_k)^\intercal \mb{t}(\mb{Y}_k)|}
 \frac{\mb{Y}_{k} - \mb{Y}_{k-1}}{|\mb{Y}_{k} - \mb{Y}_{k-1}|}, 
\]
\[
\int_{u_{k-\frac12}}^{u_{k+\frac12}} \partial_t \mb{Y} |\partial_u \mb{Y}| d u
 \approx \frac{d  \mb{Y}_k}{d t} \frac{d_{k+1} + d_k}{2}, \quad   d_k = |\mb{X}_k - \mb{X}_{k-1}|.  
\]
We approximate the partial derivatives $\frac{\partial \mathcal{X}_i}{\partial Y_j}, i=1,2,3, j=1,2,$ by   means   of finite diferences. The $2\times 2$ matrix $\nabla \mathcal{X}(\mb{Y}) \nabla \mathcal{X}(\mb{Y})^\intercal$ can be explicitly inverted. Notice that this matrix is regular because 
$\mathcal{X}$ is assumed to be an immersion. 

\end{remark}

\section{Examples}
In our examples, we consider the evolution of an initial Fourier curve parameterized  by a finite trigonometric  series in the parameter $u\in I$. Here we remind the reader that $I$ is identified with the unit circle and $I=\mathbb{R}/\mathbb{Z}\simeq S^1$. 
In all of our numerical experiments, we use $M=200$ discretization nodes, uniform tangential redistribution, and the regularization parameter $\delta$ in the discrete approximation of curvature was set to $\delta = 10^{-5}$. 
The resulting system of ODEs (\ref{ODEs}) is numerically solved using the fourth-order explicit Runge-Kutta-Merson method, incorporating automatic time step control with a tolerance level of $10^{-3}$ (refer to \cite{0965-0393-24-3-035003}). The initial time step was selected as $4h^2$, where $h=1/M$ denotes the spatial mesh size.

\subsection{Evolution of knotted curves on torus}


As an example of a torus, we can consider immersion $\mathcal{X}: I\times I\to\mathbb{R}^3$ defined as:
\begin{equation}
\label{eq:torus}
\mathcal{X}(u, v) = \left(
(r\cos(2\pi v) + R) \sin(2\pi  u),\ 
(r\cos(2\pi v) + R) \cos(2\pi u),\ 
r\sin(2\pi v) \right)^\intercal,
\end{equation}
where $0<r < R $ and $u,v\in I$.
The torus surface can also be defined as an embedded manifold ${\mathcal{M}}=\{ \mb{X}=(X_1,X_2,X_3)^\intercal\in\mathbb{R}^3,\ f(\mb{X}) =0\}$ where 
\[
f(\mb{X}) = ( (X_1^2+X_2^2)^\frac12 -R )^2 + X_3^2 -r^2.
\]
Its gradient $\nabla f(\mb{X})$ and the  Hesse   matrix $\nabla^2 f(\mb{X})$ are  given by
\begin{eqnarray}
\nabla f(\mb{X}) &=& 2 \mb{X} - 2 R \left( X_1/(X_1^2+X_2^2)^\frac12, \  X_2/(X_1^2+X_2^2)^\frac12,\ 0 \right)^\intercal,
\\
\vert \nabla f(\mb{X})\vert   &=& 2 \left( X_1^2+X_2^2+X_3^2 - 2R (X_1^2+X_2^2)^\frac12 +R^2\right)^\frac12,
\\
\nabla^2 f(\mb{X}) &=& 2 \mb{I} - 2 \frac{R}{(X_1^2+X_2^2)^\frac32} (X_2, -X_1,0) (X_2, -X_1,0)^\intercal,
\\
\mb{T}^\intercal\nabla^2 f(\mb{X}) \mb{T} &=& 2 - 2 \frac{R}{(X_1^2+X_2^2)^\frac32} (T_1 X_2 - T_2 X_1)^2,
\end{eqnarray}
where the unit tangent vector $\mb{T}=(T_1,T_2,T_3)^\intercal$.
The torus surface is shown in Fig.~\ref{fig-curve-on-torus}. The initial curve $\mb{X}_0$ derived from the mapping (\ref{eq:torus}) is parameterized by
\begin{equation}
\mb{X}_0(u) = \mathcal{X}(k u, l u), \quad u\in I,
\label{parameterization}
\end{equation}
where $k=2, l=3$, $r=1$ and $R=4$. Its time evolution is shown in Fig.~\ref{fig-curve-on-torus-2-3}. The curve shrinks and converges to the stationary state with constant length as suggested in Fig.~\ref{fig-curve-on-torus-length} a).

\begin{figure}
\begin{center}

\includegraphics[width=0.35\textwidth]{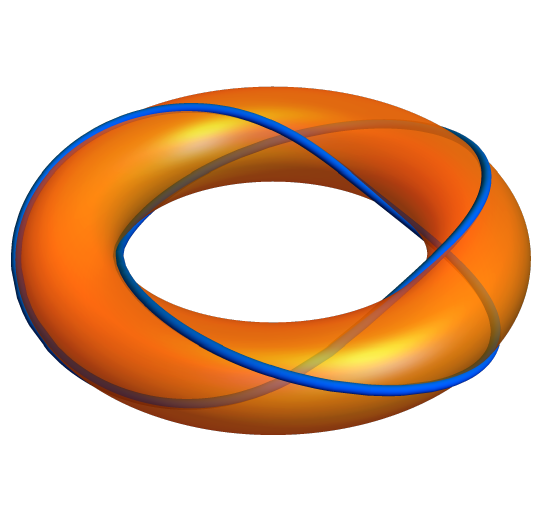}
\quad
\includegraphics[width=0.3\textwidth]{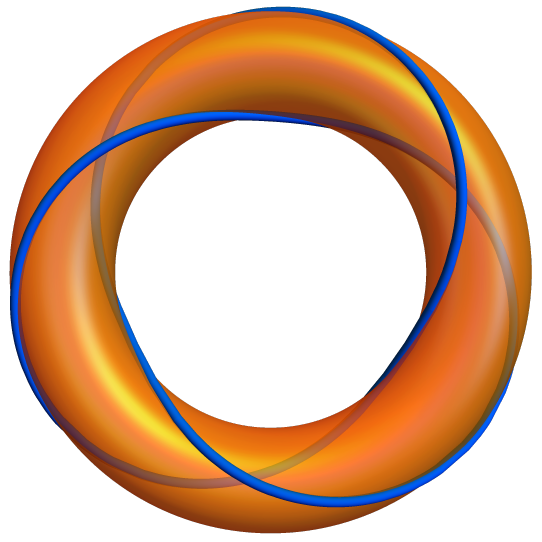}

{\small Initial curve $t=0$}
\medskip

\includegraphics[width=0.35\textwidth]{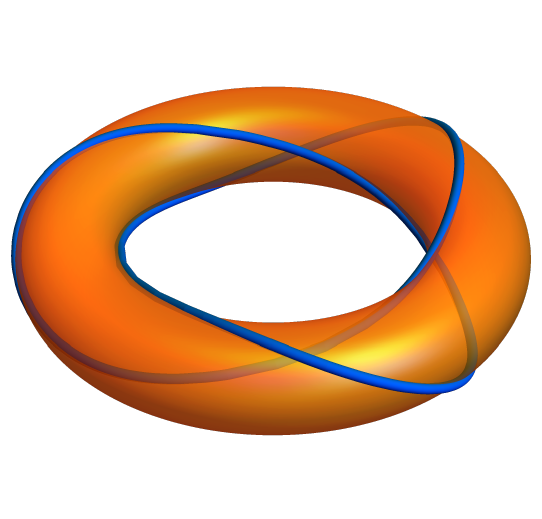}
\quad
\includegraphics[width=0.3\textwidth]{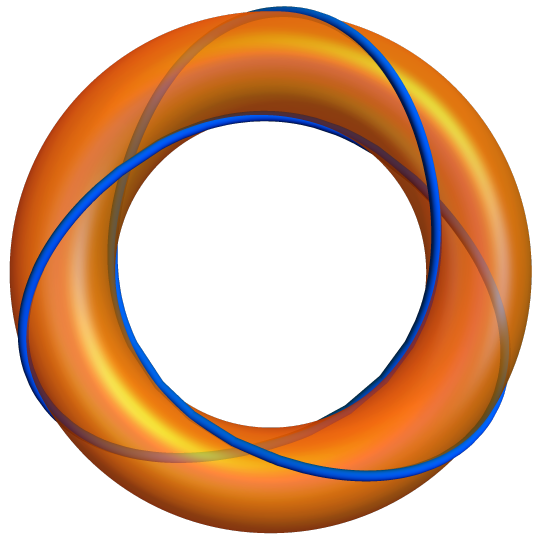}

{\small Intermediate curve $t=2.5$}
\medskip

\includegraphics[width=0.35\textwidth]{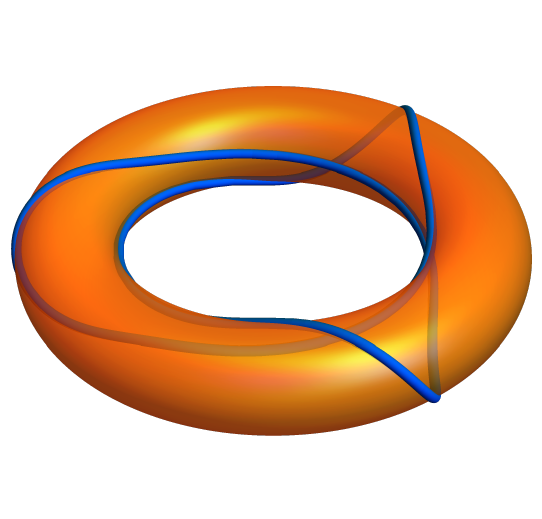}
\quad
\includegraphics[width=0.3\textwidth]{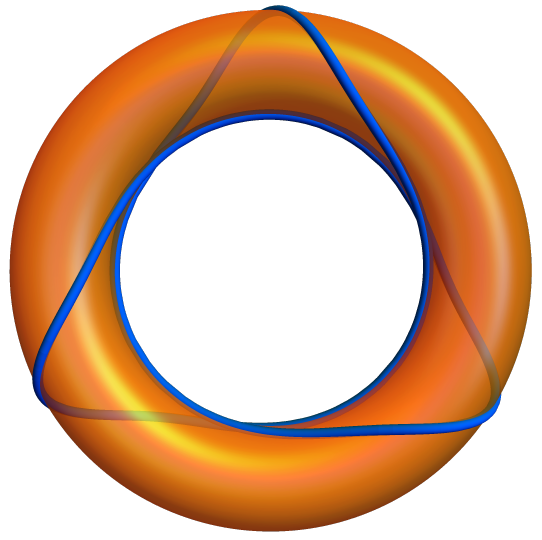}

{\small Limiting stationary curve $t=22.5$}

\caption{Time evolution of a knotted Fourier curve belonging to the orientable torus surface with parameters   $r=1, R=4$,   and  $k=2$, $l=3$  (see (\ref{parameterization})). } 
\label{fig-curve-on-torus-2-3}
\end{center}
\end{figure}

\subsection{Attraction of curves by a torus surface}

In this part, we present an example of the evolution of initial closed curves belonging to a small neighborhood of the given surface $\mathcal{M}$. 

The following computational example demonstrates how an initially knotted curve evolves according to the geometric evolution equation (\ref{eq:ab}) driven by the force term (\ref{F}). 
The reference surface is the torus given by immersion (\ref{eq:torus}) with $r = 1$, $R=4$. The initial curve $\mb{X}_0$ is parameterized by mapping (\ref{eq:torus}) as 
\[
\mb{X}_0(u) = \mathcal{X}(k u, l u), \quad u\in I,
\]
where $k=3, l=5$, $r=2$ and $R=4$. The time evolution of such an inflated initial curve is shown in Fig.~\ref{fig-curve-attract}. The curve continues to shrink until it attaches to the torus surface and eventually finds its stationary state with constant length, as suggested in Fig.~\ref{fig-curve-on-torus-length} b).

\begin{figure}
\begin{center}

\includegraphics[width=0.35\textwidth]{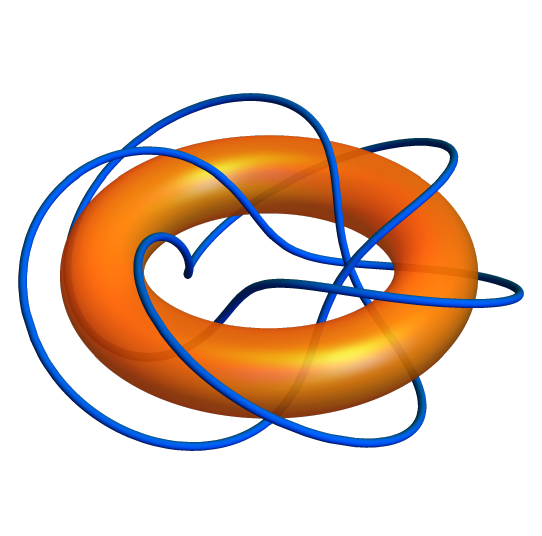}
\quad
\includegraphics[width=0.3\textwidth]{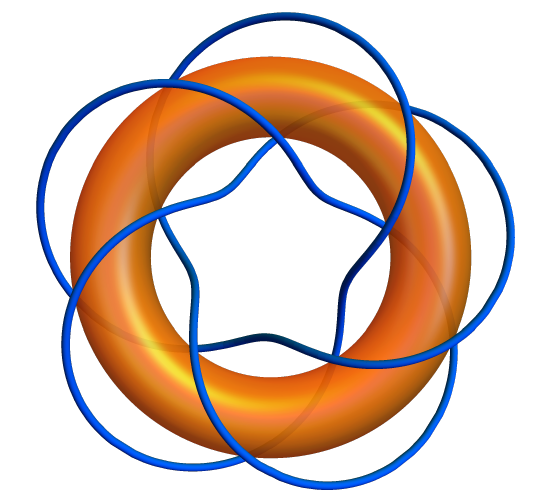}

{\small Initial curve $t=0$}
\medskip

\includegraphics[width=0.35\textwidth]{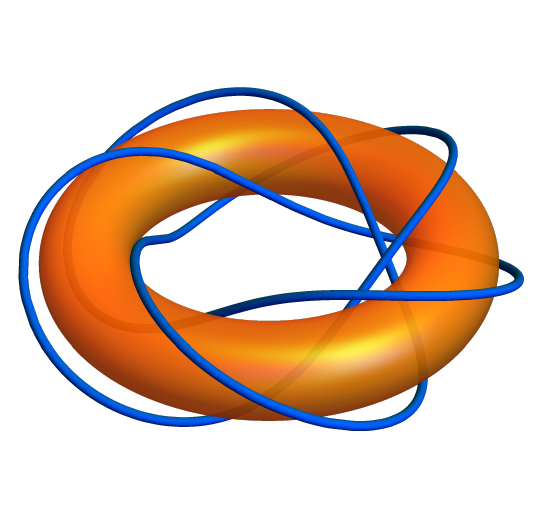}
\quad
\includegraphics[width=0.3\textwidth]{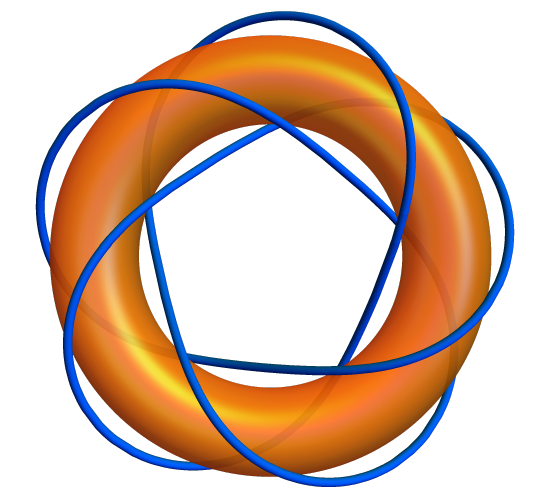}

{\small Intermediate curve $t=1$}
\medskip

\includegraphics[width=0.35\textwidth]{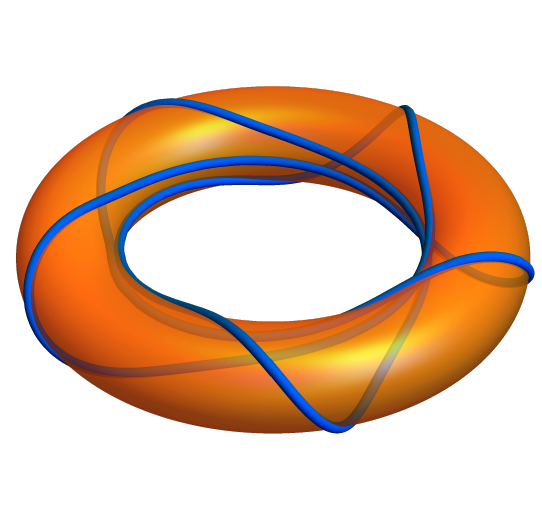}
\quad
\includegraphics[width=0.3\textwidth]{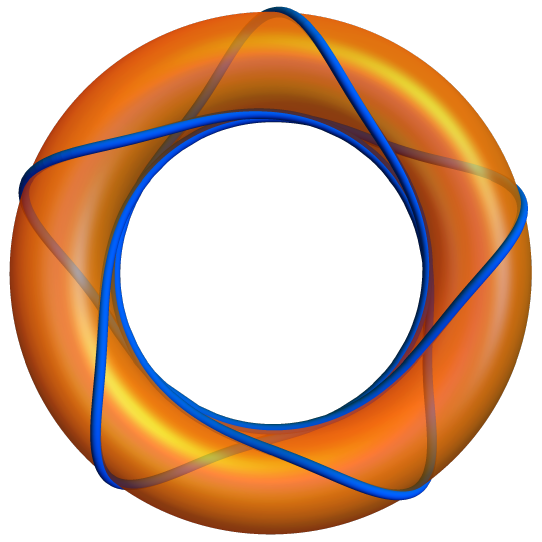}

{\small Limiting stationary curve $t=19.75$}

\caption{Time evolution of initially inflated knotted curve belonging to the orientable torus surface with parameters   $r=1, R=4$.   Topologically more complex case corresponding to the choice $k=3$, $l=5$ in initial condition   (\ref{parameterization}).}

\label{fig-curve-attract}
\end{center}
\end{figure}

\begin{figure}
\begin{center}

\includegraphics[width=0.42\textwidth]{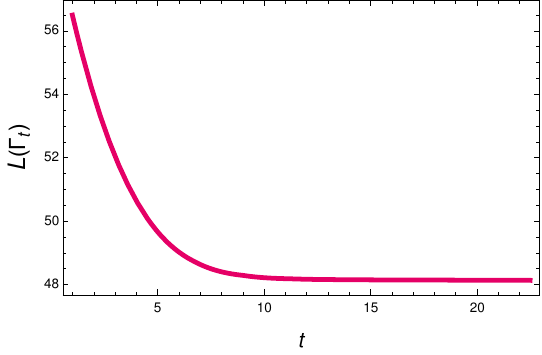}
\quad 
\includegraphics[width=0.42\textwidth]{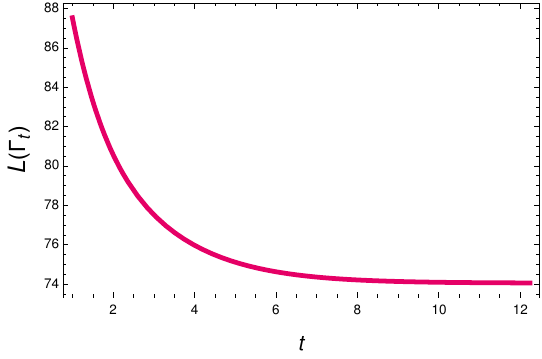}

{\small a) \hskip 7.5truecm b)}

\caption{Lengths $L(\Gamma_t)$ of shrinking curves on the torus (\ref{eq:torus}) - a) and attaching to the torus (\ref{eq:torus}) - b).} 

\label{fig-curve-on-torus-length}
\end{center}
\end{figure}

\begin{figure}
\begin{center}

\includegraphics[width=0.32\textwidth]{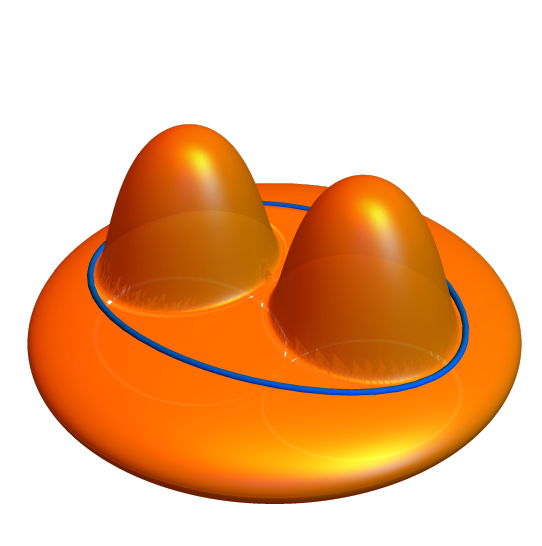}
\quad
\includegraphics[width=0.3\textwidth]{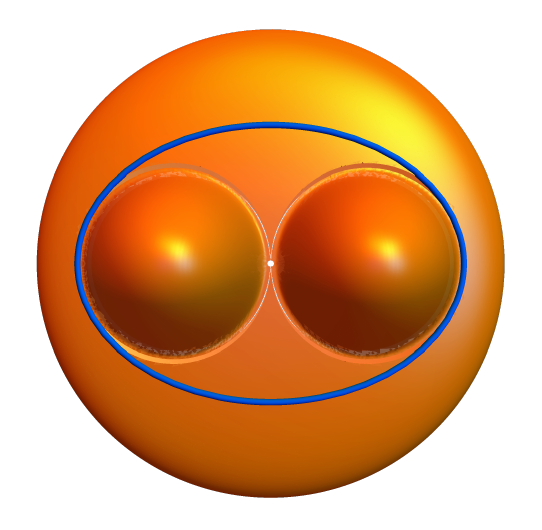}

{\small Initial curve $t=0$}
\medskip

\includegraphics[width=0.32\textwidth]{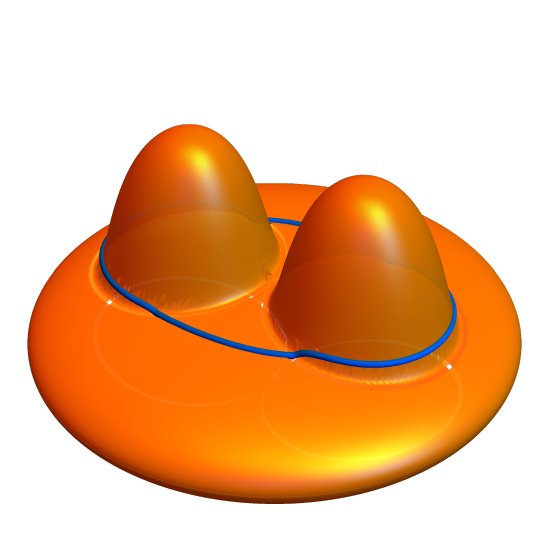}
\quad
\includegraphics[width=0.3\textwidth]{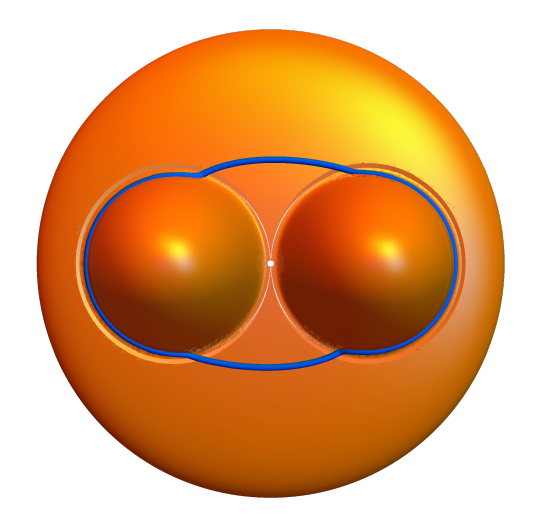}

{\small Intermediate curve $t=4.5$}
\medskip

\includegraphics[width=0.32\textwidth]{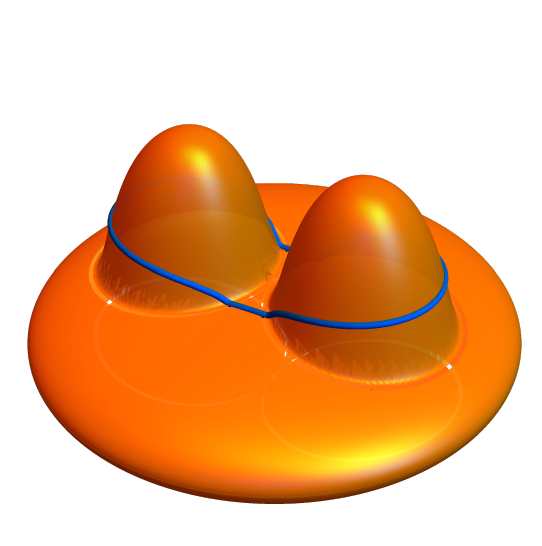}
\quad
\includegraphics[width=0.3\textwidth]{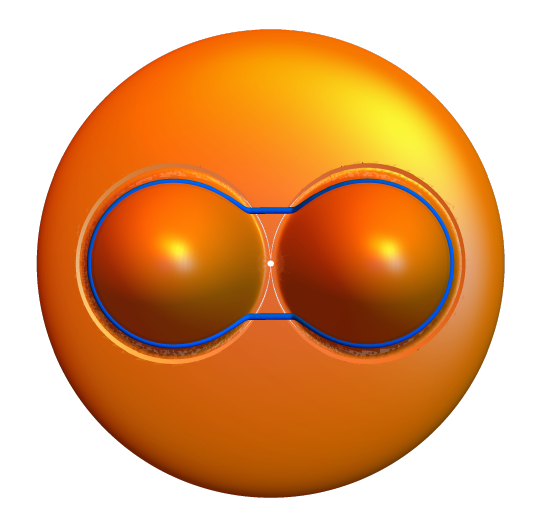}

{\small Limiting stationary curve $t=13.5$}

\caption{Evolution of a simple curve belonging to the orientable surface of genus 0 with humps.}

\label{fig-curve-on-dvakopce}
\end{center}
\end{figure}

\subsection{Evolution of simple curves on surface with  genus   0  with humps}

The last example is shown in Fig.~\ref{fig-curve-on-dvakopce}. 
${\mathcal{M}}=\{ \mb{X}=(X_1,X_2,X_3)^\intercal\in\mathbb{R}^3,\ f(\mb{X}) =0\}$ where 
\[
f(\mb{X}) = X_1^2+X_2^2 + c^2 (X_3-\phi(X_1, X_2))^2 -r^2, 
\quad \phi(X_1, X_2) = h(X_1-1, X_2)  + h(X_1+1, X_2).
\]
 Here   $h$ is a smooth bump function, $h(X_1, X_2) = v\, 2^{-1/(1-X_1^2 - X_2^2)}$ for $X_1^2+X_2^2<1$, and $h(X_1, X_2) = 0$, otherwise. In the example shown in Fig.~\ref{fig-curve-on-dvakopce} we set $r=2.5, c=4$, and $v=3$. The initial curve is an ellipse projected onto the surface, i.e.
\[
\Gamma_0=\{ \mb{X}=(X_1, X_2, X_3), \quad X_1^2/2 + X_2^2 = 2, \quad X_3= 
c^{-1} (r^2 - X_1^2 -X_2^2)^\frac12 + \phi(X_1, X_2)\}.
\]

\section{Conclusion}

In this paper, we investigated the curvature-driven flow of a family of closed curves evolving on an embedded or immersed manifold in the three-dimensional Euclidean space. We analyzed the qualitative behavior of this flow. Using the abstract theory of analytic semi-flows in Banach spaces, we prove the local existence and uniqueness of H\"older smooth solutions to the governing system of evolution equations for the curve parametrization. We demonstrate the behavior of sulutions in several computational examples constructed by means of the flowing finite-volume method and asymptotically uniform tangential redistribution of discretization points. 

\bigskip
\section*{Declarations}
\noindent {\bf  Funding:}  D. \v{S}ev\v{c}ovi\v{c} was supported by the VEGA 1-0493-24 research project.
M. Kol\'a\v{r} was supported by Czech Science Foundation Project no. 25-18265S - Computational models of hydraulic fracturing in geothermal energy production.

\noindent {\bf Conflict of interest:} The authors declare that there are no conflicts of interest with respect to the publication of this paper

\noindent {\bf Authors' contributions:} The authors contributed equally to the study.

\noindent{\bf Ethical approval:} Not applicable.

\noindent{\bf Data availability:} All data in this paper are available from the author on a reasonable request.

\bibliographystyle{siam}

\bibliography{paper}

@article{Angenent1990b,
  title={Parabolic equations for curves on surfaces Part I. Curves with p-integrable curvature},
  author={Angenent, Sigurd},
  journal={Annals of Mathematics},
  pages={451--483},
  year={1990},
  publisher={JSTOR}
}

@article{Angenent1990,
  title={Nonlinear analytic semiflows},
  author={Angenent, Sigurd B},
  journal={Proceedings of the Royal Society of Edinburgh Section A: Mathematics},
  volume={115},
  number={1-2},
  pages={91--107},
  year={1990},
  publisher={Royal Society of Edinburgh Scotland Foundation}
}

@article{BKS2020,
  title={Curvature driven flow of a family of interacting curves with applications},
  author={Bene{\v{s}}, Michal and Kol{\'a}{\v{r}}, Miroslav and {\v{S}}ev{\v{c}}ovi{\v{c}}, Daniel},
  journal={Mathematical Methods in the Applied Sciences},
  volume={43},
  number={7},
  pages={4177--4190},
  year={2020},
  publisher={Wiley Online Library}
}

@article{daprato,
  title={Equations d’{\'e}volution abstraites non lin{\'e}aires de type parabolique},
  author={Grisvard, G Da Prato-P and Da Prato, G},
  journal={Ann. Mat. Pura Appl.,(4)},
  volume={120},
  pages={329--396},
  year={1979}
}

@article{Deckelnick2025,
  title={Discrete anisotropic curve shortening flow in higher codimension},
  author={Deckelnick, Klaus and N{\"u}rnberg, Robert},
  journal={IMA Journal of Numerical Analysis},
  volume={45},
  number={1},
  pages={36--67},
  year={2025},
  publisher={Oxford University Press}
}

@article{Helmholtz1858,
 author = {Helmholtz, H.},
 title = {{\"U}ber {Integrale} der hydrodynamischen {Gleichungen}, welche den {Wirbelbewegungen} entsprechen.},
 fjournal = {Journal f{\"u}r die Reine und Angewandte Mathematik},
 journal = {J. Reine Angew. Math.},
 issn = {0075-4102},
 volume = {55},
 pages = {25--55},
 year = {1858},
 language = {German},
 doi = {10.1515/crll.1858.55.25},
 url = {https://eudml.org/doc/147720},
 zbMATH = {2750339},
 ERAM = {055.1448cj}
}

@article{0965-0393-24-3-035003,
  title={Dynamics of dislocations described as evolving curves interacting with obstacles},
  author={Pau{\v{s}}, Petr and Bene{\v{s}}, Michal and Kol{\'a}{\v{r}}, Miroslav and Kratochv{\'\i}l, Jan},
  journal={Modelling and Simulation in Materials Science and Engineering},
  volume={24},
  number={3},
  pages={035003},
  year={2016},
  publisher={IOP Publishing}
}

@article{kolar2022,
  title={Qualitative and numerical aspects of a motion of a family of interacting curves in space},
  author={Bene{\v{s}}, Michal and Kol{\'a}{\v{r}}, Miroslav and {\v{S}}ev{\v{c}}ovi{\v{c}}, Daniel},
  journal={SIAM journal on applied mathematics},
  volume={82},
  number={2},
  pages={549--575},
  year={2022},
  publisher={SIAM}
}

@article{benes2024diffusion,
  title={On diffusion and transport acting on parameterized moving closed curves in space},
  author={Bene{\v{s}}, Michal and Kol{\'a}{\v{r}}, Miroslav and {\v{S}}ev{\v{c}}ovi{\v{c}}, Daniel},
  journal={Mathematical Methods in the Applied Sciences},
  volume={},
  number={},
  pages={1--15},
  year={2025},
 doi = {10.13001/ela.2024.7951},
}

@article{Binz,
  title={A convergent finite element algorithm for mean curvature flow in arbitrary codimension},
  author={Binz, Tim and Kov{\'a}cs, Bal{\'a}zs},
  journal={Interfaces and Free Boundaries},
  volume={25},
  number={3},
  pages={373--400},
  year={2023}
}

@article{MS3,
  title={Evolution of curves on a surface driven by the geodesic curvature and external force},
  author={Mikula, Karol and {\v{S}}ev{\v{c}}ovi{\v{c}}, Daniel},
  journal={Applicable Analysis},
  volume={85},
  number={4},
  pages={345--362},
  year={2006},
  publisher={Taylor \& Francis}
}

@article{Kolar_algoritmy,
	author = {Miroslav Kolář and Daniel Ševčovič},
	title = { Evolution of multiple closed knotted curves in space},
	journal = {Proceedings of the Conference Algoritmy},


	year = {2024},
	
	pages = {129--138},	
}

@article{PBKK:21,
  title={Improving method for deterministic treatment of double cross-slip in FCC metals under low homologous temperatures},
  author={Kol{\'a}{\v{r}}, Miroslav and Pau{\v{s}}, Petr and Kratochv{\'\i}l, Jan and Bene{\v{s}}, Michal},
  journal={Computational Materials Science},
  volume={189},
  pages={110251},
  year={2021},
  publisher={Elsevier}
}

@article{Reneker,
  title={Electrospinning jets and polymer nanofibers},
  author={Reneker, Darrell H and Yarin, Alexander L},
  journal={Polymer},
  volume={49},
  number={10},
  pages={2387--2425},
  year={2008},
  publisher={Elsevier}
}

@article{Yarin,
  title={Fundamentals and applications of micro-and nanofibers},
  author={Gilmore, Charles M},
  journal={MRS BULLETIN},
  volume={40},
  pages={87--88},
  year={2015}
}

@book{He,
  title={Electrospun nanofibres and their applications},
  author={He, Ji-Huan and Liu, Yong and Mo, Lu-Feng and Wan, Yu-Qin and Xu, Lan},
  year={2008},
  publisher={ISmithers Shawbury, UK}
}

@article{Xu,
  title={Numerical simulation of a two-phase flow in the electrospinning process},
  author={Xu, Lan and Liu, HongYing and Si, Na and Wai Ming Lee, Eric},
  journal={International Journal of Numerical Methods for Heat \& Fluid Flow},
  volume={24},
  number={8},
  pages={1755--1761},
  year={2014},
  publisher={Emerald Group Publishing Limited}
}

@article{sevcovic2001evolution,
  title={Evolution of plane curves driven by a nonlinear function of curvature and anisotropy},
  author={Mikula, Karol and {\v{S}}ev{\v{c}}ovi{\v{c}}, Daniel},
  journal={SIAM Journal on Applied Mathematics},
  volume={61},
  number={5},
  pages={1473--1501},
  year={2001},
  publisher={SIAM}
}

@article{MS2004,
  title={Computational and qualitative aspects of evolution of curves driven by curvature and external force},
  author={Mikula, Karol and {\v{S}}ev{\v{c}}ovi{\v{c}}, Daniel},
  journal={Computing and Visualization in Science},
  volume={6},
  number={4},
  pages={211--225},
  year={2004},
  publisher={Springer}
}

@article{MMAS2004,
  title={A direct method for solving an anisotropic mean curvature flow of plane curves with an external force},
  author={Mikula, Karol and {\v{S}}ev{\v{c}}ovi{\v{c}}, Daniel},
  journal={Mathematical Methods in the Applied Sciences},
  volume={27},
  number={13},
  pages={1545--1565},
  year={2004},
  publisher={Wiley Online Library}
}

@article{Hou,
  title={Removing the stiffness from interfacial flows with surface tension},
  author={Hou, Thomas Y and Lowengrub, John S and Shelley, Michael J},
  journal={Journal of Computational Physics},
  volume={114},
  number={2},
  pages={312--338},
  year={1994},
  publisher={Elsevier}
}

@article{Kimura,
  title={Numerical analysis of moving boundary problems using the boundary tracking method},
  author={Kimura, Masato},
  journal={Japan journal of industrial and applied mathematics},
  volume={14},
  pages={373--398},
  year={1997},
  publisher={Springer}
}

@article{BKS2017,
  title={Area preserving geodesic curvature driven flow of closed curves on a surface},
  author={Bene{\v{s}}, Michal and Kol{\'a}{\v{r}}, Miroslav and {\v{S}}ev{\v{c}}ovi{\v{c}}, Daniel},
  journal={Discrete Contin. Dyn. Syst. Ser. B},
  volume={22},
  number={10},
  pages={3671--3689},
  year={2017}
}

@article{Barret2010,
  title={Numerical approximation of gradient flows for closed curves in {R}${}^d$},
  author={Barrett, John W and Garcke, Harald and N{\"u}rnberg, Robert},
  journal={IMA journal of numerical analysis},
  volume={30},
  number={1},
  pages={4--60},
  year={2010},
  publisher={OUP}
}

@article{Barret2012,
  title={Parametric approximation of isotropic and anisotropic elastic flow for closed and open curves},
  author={Barrett, John W and Garcke, Harald and N{\"u}rnberg, Robert},
  journal={Numerische Mathematik},
  volume={120},
  number={3},
  pages={489--542},
  year={2012},
  publisher={Springer}
}

@article{Elliot2017,
  title={On approximations of the curve shortening flow and of the mean curvature flow based on the DeTurck trick},
  author={M. Elliott, Charles and Fritz, Hans},
  journal={IMA Journal of Numerical Analysis},
  volume={37},
  number={2},
  pages={543--603},
  year={2017},
  publisher={Oxford University Press}
}

@article{Garcke2009,
  title={Nonlinear stability of stationary solutions for curvature flow with triple junction},
  author={Garcke, Harald and Kohsaka, Yoshihito and  {\v{S}}ev{\v{c}}ovi{\v{c}}, Daniel},
  journal={Hokkaido Mathematical Journal},
  volume={38},
  number={4},
  pages={721--769},
  year={2009},
  publisher={Hokkaido University, Department of Mathematics}
}

@article{MS2014,
  title={Manifold evolution with tangential redistribution of points},
  author={Mikula, Karol and Reme\v{s}\'{\i}kov\'a, Mariana and Sarkoci, Peter and {\v{S}}ev{\v{c}}ovi{\v{c}}, Daniel},
  journal={SIAM Journal on Scientific Computing},
  volume={36},
  number={4},
  pages={A1384--A1414},
  year={2014},
  publisher={SIAM}
}

@article{Lunardi1984,
  title={Abstract quasilinear parabolic equations},
  author={Lunardi, Alessandra},
  journal={Mathematische Annalen},
  volume={267},
  number={3},
  pages={395--415},
  year={1984},
  publisher={Springer}
}

@article{kolar2017area,
  title={Area preserving geodesic curvature driven flow of closed curves on a surface},
  author={Kol{\'a}r, Miroslav and Bene{\v{s}}, Michal and {\v{S}}evcovic, Daniel},
  journal={Discrete Contin. Dyn. Syst. Ser. B},
  volume={22},
  number={10},
  pages={3671--3689},
  year={2017}
}

@article{vsevvcovivc2011evolution,
  title={Evolution of plane curves with a curvature adjusted tangential velocity},
  author={{\v{S}}ev{\v{c}}ovi{\v{c}}, Daniel and Yazaki, Shigetoshi},
  journal={Japan journal of industrial and applied mathematics},
  volume={28},
  pages={413--442},
  year={2011},
  publisher={Springer}
}

@article{niu2019dislocation,
  title={Dislocation dynamics formulation for self-climb of dislocation loops by vacancy pipe diffusion},
  author={Niu, Xiaohua and Gu, Yejun and Xiang, Yang},
  journal={International Journal of Plasticity},
  volume={120},
  pages={262--277},
  year={2019},
  publisher={Elsevier}
}

@article{niu2017dislocation,
  title={Dislocation climb models from atomistic scheme to dislocation dynamics},
  author={Niu, Xiaohua and Luo, Tao and Lu, Jianfeng and Xiang, Yang},
  journal={Journal of the Mechanics and Physics of Solids},
  volume={99},
  pages={242--258},
  year={2017},
  publisher={Elsevier}
}

@article{dziuk1994convergence,
  title={Convergence of a semi-discrete scheme for the curve shortening flow},
  author={Dziuk, Gerhard},
  journal={Mathematical Models and Methods in Applied Sciences},
  volume={4},
  number={04},
  pages={589--606},
  year={1994},
  publisher={World Scientific}
}

@article{deckelnick1997weak,
  title={Weak solutions of the curve shortening flow},
  author={Deckelnick, Klaus},
  journal={Calculus of Variations and Partial Differential Equations},
  volume={5},
  number={6},
  pages={489--510},
  year={1997},
  publisher={Springer}
}

@article{gage1986heat,
  title={The heat equation shrinking convex plane curves},
  author={Gage, Michael and Hamilton, Richard S},
  journal={Journal of Differential Geometry},
  volume={23},
  number={1},
  pages={69--96},
  year={1986},
  publisher={Lehigh University}
}

@article{kemmochi2024structure,
  title={Structure-preserving numerical methods for constrained gradient flows of planar closed curves with explicit tangential velocities},
  author={Kemmochi, Tomoya and Miyatake, Yuto and Sakakibara, Koya},
  journal={Japan Journal of Industrial and Applied Mathematics},
  pages={1--29},
  year={2024},
  publisher={Springer}
}

@article{ASENS_1990_4_23_2_229_0,
     author = {Gage, Michael E.},
     title = {Curve shortening on surfaces},
     journal = {Annales scientifiques de l'\'Ecole Normale Sup\'erieure},
     pages = {229--256},
     publisher = {Elsevier},
     volume = {Ser. 4, 23},
     number = {2},
     year = {1990},
     doi = {10.24033/asens.1603},
     mrnumber = {91a:53072},
     zbl = {0713.53022},
     language = {en},
     url = {https://www.numdam.org/articles/10.24033/asens.1603/}
}

\end{document}